\documentclass[review,hidelinks,onefignum,onetabnum]{siamart250211}

\usepackage{bm}
\pretocmd{\maketitle}{\nolinenumbers}{}{}
\usepackage{makecell}
\usepackage{algorithmic}

\usepackage{lipsum}
\usepackage{amsfonts}
\usepackage{graphicx}
\usepackage{epstopdf}
\usepackage{algorithmic}
\ifpdf
  \DeclareGraphicsExtensions{.eps,.pdf,.png,.jpg}
\else
  \DeclareGraphicsExtensions{.eps}
\fi

\newsiamremark{remark}{Remark}
\newsiamremark{hypothesis}{Hypothesis}
\crefname{hypothesis}{Hypothesis}{Hypotheses}
\newsiamthm{claim}{Claim}
\newsiamremark{fact}{Fact}
\crefname{fact}{Fact}{Facts}

\headers{Adaptive first-order  gradient methods}{X. Hu, S. Pollock, Z. Xue, and Y. Zhu}

\title{An adaptive framework for \\ first-order gradient methods\thanks{Submitted to the editors DATE.
\funding{Author SP acknowledges partial support from NSF grant DMS 2045059 (CAREER). The work of XH is partially supported by the National Science Foundation under grant DMS-2513394.}}}

\author{Xiaozhe Hu \thanks{Department of Mathematics, Tufts University, Medford, MA 02155, USA
   (\email{Xiaozhe.Hu@tufts.edu},
   \email{Zhongqin.Xue@tufts.edu}).}
\and Sara Pollock\thanks{Department of Mathematics, University of Florida, Gainesville, FL 32611, USA 
    (\email{s.pollock@ufl.edu}).}
  \and Zhongqin Xue\footnotemark[2]
  \and Yunrong Zhu\thanks{Department of Mathematics and Statistics, Idaho State University, Pocatello, ID 83209, USA 
    (\email{zhuyunr@isu.edu}).}
}

\usepackage{amsopn}

\ifpdf
\hypersetup{
  pdftitle={Adaptive framework for gradient methods},
  pdfauthor={X. Hu, S. Pollock, Z. Xue, and Y. Zhu}
}
\fi

\begin{document}

\maketitle

\begin{abstract}
Gradient methods are widely used in optimization problems. In practice, while the smoothness parameter can be estimated utilizing techniques such as backtracking, estimating the strong convexity parameter remains a challenge; moreover, even with the optimal parameter choice, convergence can be slow. In this work, we propose a framework for dynamically adapting the step size and momentum parameters in first-order gradient methods for the optimization problem, without prior knowledge of the strong convexity parameter. The main idea is to use the geometric average of the ratios of successive residual norms as an empirical estimate of the upper bound on the convergence rate, which in turn allows us to adaptively update the algorithm parameters. The resulting algorithms are simple to implement, yet efficient in practice, requiring only a few additional computations on existing information. The proposed adaptive gradient methods are shown to converge at least as fast as gradient descent for quadratic optimization problems. Numerical experiments on both quadratic and nonlinear problems validate the effectiveness of the proposed adaptive algorithms. The results show that the adaptive algorithms are comparable to their counterparts using optimal parameters, and in some cases, they capture local information and exhibit improved performance.
\end{abstract}



\section{Introduction}
Many problems in machine learning \cite{chong2023introduction,sun2019survey,bottou2018optimization}, signal processing \cite{mattingley2010real,palomar2010convex}, and operations research \cite{meignan2015review,rardin1998optimization}  are formulated as optimization problems: 
\begin{align*}
    \min_{\bm{x} \in \mathbb{R}^n} f(\bm{x}),
\end{align*}
where we assume that $f:\mathbb{R}^n \to \mathbb{R}$ is a differentiable function. The function $f$ is said to be $L$-smooth if its gradient is Lipschitz continuous with constant $L>0$, that is, 
\begin{align*}
    \|\nabla f(\bm{x}) - \nabla f(\bm{y})\| \leq L \|\bm{x} - \bm{y}\|, \quad \forall \, \bm{x}, \bm{y} \in \mathbb{R}^n.
\end{align*}
The function $f$ is convex if it satisfies
\begin{align*}
  f(\bm{y}) \ge f(\bm{x}) + \langle \nabla f(\bm{x}), \bm{y} - \bm{x} \rangle, \quad \forall \, \bm{x}, \bm{y} \in \mathbb{R}^n.
\end{align*}
Furthermore, $f$ is said to be $\mu$-strongly convex with parameter $\mu > 0$ if
\begin{align*}
  f(\bm{y}) \geq f(\bm{x}) + \langle \nabla f(\bm{x}), \bm{y} - \bm{x} \rangle + \frac{\mu}{2} \|\bm{y} - \bm{x}\|^2, \quad \forall \, \bm{x}, \bm{y} \in \mathbb{R}^n.
\end{align*}
One of the most  commonly used methods for solving such problems is the gradient descent (GD) algorithm. When $f$ is convex and $L$-smooth, GD achieves a sublinear convergence, which varies with the smoothness of the gradient. For $\mu$-strongly convex problems, the convergence rate of GD improves to linear. However, GD often exhibits slow convergence for ill-conditioned problems. To accelerate convergence, first-order momentum methods were introduced. The Nesterov Accelerated Gradient (NAG)  utilizes a ``look-ahead" gradient by evaluating $\nabla f$ at an extrapolated point  achieving an accelerated linear convergence rate compared to GD. By incorporating momentum directly, the heavy-ball (HB) method achieves acceleration for iterations near the minimum of the strongly convex quadratic objective; however, for general strongly convex problems, it may fail to converge \cite{goujaud2025provable,lessard2016analysis}. 

However, in practice, smoothness and strong convexity constants are typically unavailable or costly to obtain in a timely manner during the optimization process. Consequently, much effort has been devoted to designing algorithms that are robust under estimated or unknown parameters. For algorithms with estimated parameters, however, poor estimations may result in slow convergence or even divergence.  Backtracking line search methods, originating from the early work of Goldstein \cite{goldstein1962cauchy} and Armijo \cite{armijo1966minimization}, 
find a step size along the search direction that ensures sufficient decrease of the objective at each iteration. 
Nesterov \cite{nesterov1983} adapts the backtracking idea to estimate the smoothness constant in constructing the accelerated gradient method. On the other hand, restart schemes provide a simple adaptive mechanism that does not require prior knowledge of the problem parameters. Their performance is primarily driven by the choice of restart frequency and the underlying regularity properties. Nesterov \cite{nesterov2013gradient} designed a restart strategy with linear convergence in the strongly convex setting. Based on the observation that accelerated methods display periodic oscillations when momentum is high, \cite{o2015adaptive} restarts the algorithm when periodic oscillations are detected.
There are also some works that investigate restart schemes under the generic H\"olderian error bound setting, see \cite{d2021acceleration,ito2021nearly,li2018calculus,roulet2017sharpness}  for details. Recently, various adaptive strategies are applied to select parameters. Malitsky and Mishchenko \cite{malitsky2019adaptive} develop an adaptive GD algorithm that estimates local smoothness from gradient information. Utilizing primal–dual relations to estimate a local strong convexity constant, they further propose an adaptive accelerated gradient descent algorithm. The authors later extend their approach to construct an adaptive proximal gradient method \cite{malitsky2024adaptive}. 
Recently, the authors in \cite{suh2025adaptive} develop an adaptive NAG that estimates the local smoothness constant via the quadratic upper bound for the $L$-smooth objective, achieving an $\mathcal{O}(1/k^2)$ convergence rate.

In this work, we illustrate our approach using an unconstrained quadratic optimization problem:
\begin{align}\label{quadra_problem}
    \min _{\bm{x} \in \mathbb{R}^n} f(\bm{x}) = \frac{1}{2} \bm{x}^TA\bm{x} - \bm{b}^T\bm{x},
\end{align}
where $A \in \mathbb{R}^{n \times n}$ is a symmetric positive definite (SPD) matrix and vector $\bm{b}\in \mathbb{R}^n$ is given. 
The function $f(\bm{x})$ is strongly convex, smooth and has a unique minimizer $\bm{x}^*$ characterized by $ \nabla f(\bm{x}^*) = 0 $, i.e., $A\bm{x}^* = \bm{b}$.  Suppose the eigenvalues of $A$ satisfy $0<\lambda_1 \leq \cdots \leq \lambda_n$. 
Then the strong convexity and smoothness parameters are given by $\mu = \lambda_1$ and $L = \lambda_n$, respectively, and the condition number is defined as $\kappa := \frac{L}{\mu}$. As discussed above, the smoothness parameter can be estimated via backtracking.  Consequently, the main interest of this work is providing a unified framework for developing adaptive methods that do not require prior knowledge of the strong convexity constant $\mu$. To this end, we introduce an adaptive update rule motivated by \cite{APZ24} that uses the geometric average of the ratios of successive residual norms as an empirical estimate of the upper bounds on convergence rates for the first-order gradient methods, which then guides the parameter updates. We prove that the adaptive gradient methods converge at least as fast as GD with a step size of $1/L$.
Numerical results on quadratic problems with diagonal Hessians validate the effectiveness of the proposed adaptive algorithms.  
We observed that the convergence rate estimate computed from the geometric average of the ratios of successive residual norms asymptotically approaches the optimal convergence rate.
Furthermore, the proposed methods, capturing local curvature
information, also show good performance on nonlinear  problems. 

The remainder of this article is organized as follows. In Section \ref{adaptive_ag}, we briefly summarize relevant gradient methods and present and analyze the proposed adaptive algorithms. Section \ref{numeri_test} reports our numerical results, evaluating the adaptive methods on both quadratic and nonlinear problems. Finally, Section \ref{conclusion} provides concluding remarks and outlines directions for future work.




\section{Adaptive Algorithms}\label{adaptive_ag}
In this section, we present  adaptive algorithms for solving the quadratic optimization problem \eqref{quadra_problem} and investigate their convergence behavior in detail.

To facilitate the discussion, we introduce some standard notation. We use $\rho^*$ as a generic notation for the upper bound on the convergence rate of the gradient methods. The parameters $\alpha$ and $\beta$ denote the step size and the momentum parameter, respectively. 
When referring to methods, the corresponding bounds will be denoted by appropriate subscripts. 
We start with a brief review of three classical first-order algorithms: GD, NAG and HB for quadratic optimization problems in Table \ref{tab:classic_methods}.  

\begin{table}[htbp]
\centering
\caption{Classical first-order gradient methods for quadratic optimization problems}
\begin{tabular}{c|c|c|c}
\hline
Method& Update Rule & Optimal Parameters&  Rate Bound \\
\hline
GD & 
$ \bm{x}_{k+1} = \bm{x}_k - \alpha \nabla f(\bm{x}_k)$ & 
$ \alpha_{\texttt{GD}}^* = \frac{2}{L+\mu} $ & 
$ \rho_{\texttt{GD}}^*= \frac{L - \mu}{L + \mu} $ \\
\hline
NAG & 
\makecell{
$\bm{y}_{k}=\bm{x}_{k}+\beta\left(\bm{x}_{k}-\bm{x}_{k-1}\right) $ \\ 
$ \bm{x}_{k+1}=\bm{y}_k-\alpha \nabla f\left(\bm{y}_k\right) $
} & 
\makecell{
$ \alpha_{\texttt{NAG}}^* = \frac{1}{L} $, \\ 
$ \beta_{\texttt{NAG}}^* = \frac{\sqrt{L} - \sqrt{\mu}}{\sqrt{L} + \sqrt{\mu}}$
} & 
$ \rho_{\texttt{NAG}}^*= 1 - \frac{\sqrt{\mu}}{\sqrt{L}} $ \\
\hline
HB & 
\makecell{
$ \bm{x}_{k+1}=\bm{x}_k-\alpha \nabla f\left(\bm{x}_k\right)$\\+$\beta\left(\bm{x}_k-\bm{x}_{k-1}\right)$
} & 
\makecell{
$\alpha_{\texttt{HB}}^* = \frac{4}{(\sqrt{L} + \sqrt{\mu})^2} $\\ $\beta_{\texttt{HB}}^* = \left(\frac{\sqrt{L} - \sqrt{\mu}}{\sqrt{L} + \sqrt{\mu}}\right)^2$
} & 
$ \rho_{\texttt{HB}}^*=\frac{\sqrt{L}-\sqrt{\mu}}{\sqrt{L}+\sqrt{\mu}} $ \\
\hline
\end{tabular}
\label{tab:classic_methods}
\end{table}

Hereafter, we assume, without loss of generality, that $L\leq1$. This can be achieved by applying a proper rescaling of $A$. Specifically, the case $L=1$ corresponds to knowing the smoothness constant of the objective in the rescaling step. 
The key idea of the algorithm is to adaptively update the step size $\alpha$ and momentum parameter $\beta$ based on the information extracted from the observed iterates. In particular, estimates of the convergence rate $\rho^*$ are used to update the parameters based on their relationships outlined in Table~\ref{tab:classic_methods}. We begin by applying this principle to GD and derive an adaptive variant in the next subsection.


\subsection{Adaptive Gradient Descent}

Define the residual at the $k$-th iteration as $\bm{r}_k := \bm{b} - A \bm{x}_k = A(\bm{x}^* - \bm{x}_k)$. For GD, the residual sequence satisfies the following recurrence
\begin{align*}
    \bm{r}_{k} = (I - \alpha A) \bm{r}_{k-1},\quad k\geq 1.
\end{align*}
Applying the $\ell
^2$ norm on both sides, noting that $I-\alpha A $ is SPD, one arrives at
\begin{align*}
    \|\bm{r}_{k}\| \leq \rho(I-\alpha A)\|\bm{r}_{k-1}\|,
\end{align*}
where $\rho(M)$ denotes the spectral radius of a matrix $M$. When $\alpha=\alpha_{\texttt{GD}}^*$, we obtain $\rho(I-\alpha A)=\rho_{\texttt{GD}}^* = \frac{L-\mu}{L+\mu}$. Therefore, we can also rewrite it as
\begin{align}\label{mu_rho_GD}
    \mu = L\frac{1-\rho_{\texttt{GD}}^*}{1+\rho_{\texttt{GD}}^*}.
\end{align}
In our adaptive scheme, we simply approximate $L$ as $1$. This is justified because estimating the largest eigenvalue is generally computationally inexpensive, allowing us to rescale the system such that $L\approx 1$.  The key idea of our proposed adaptive framework is to construct a sequence $\{\rho_k\}$ that serves as an empirical estimate of $\rho_{\texttt{GD}}^{*}$ at each iteration. 
Then, \eqref{mu_rho_GD} suggests that we could approximate $\mu$ at each iteration by
$\mu_{k} = \frac{1 - \rho_{k}}{1 + \rho_{k}}$, which yields an adaptive step size:
\begin{align*}
\alpha_{\texttt{GD}}^* = \frac{2}{L+\mu}  \Longrightarrow   \alpha_{k+1} = \frac{2}{1 + \mu_{k}}=1+\rho_k.
\end{align*}
The update rule for the adaptive GD is then given by $$\bm{x}_{k+1} = \bm{x}_{k} - \alpha_{k+1} \nabla f(\bm{x}_{k}).$$
It is summarized in Algorithm \ref{alg:AGD}.

\begin{algorithm}[h!]
\caption{Adaptive Gradient Descent }
\begin{algorithmic}[1]\label{alg:AGD}
\STATE \textbf{Input:} SPD matrix $A \in \mathbb{R}^{n \times n}$, vector $\bm{b} \in \mathbb{R}^n$, initial point $\bm{x}_0 \in \mathbb{R}^n$
\STATE \textbf{Initialize:} Compute $\bm{r}_0 = -\nabla f(\bm{x}_0)$, set $\alpha_1 = 1$
\FOR{$k = 1, 2, \dots$}
    \STATE $\bm{x}_{k} = \bm{x}_{k-1} + \alpha_k \bm{r}_{k-1}$
    \STATE $\bm{r}_{k} = - \nabla f(\bm{x}_{k}) $
    \STATE $\rho_k \approx \rho_{\texttt{GD}}^*,\text{ where } \rho_k \text{ is an estimate of } \rho_{\texttt{GD}}^*$
    \STATE $\alpha_{k+1} = 1+\rho_k$
\ENDFOR
\end{algorithmic}
\end{algorithm}

First, we consider the case where the estimate $\rho_k$ is obtained by computing the spectral radius of $I - \alpha_k A$. Note that this case is primarily of theoretical interest, as it is computationally expensive to implement in practice.
\begin{theorem} \label{thm:GD-rhok}
If $L < 1$, let $\rho_k = \rho(I - \alpha_k A)$ in \Cref{alg:AGD}. Then $\rho_k$ satisfies
\begin{align*}
    \lim\limits_{k \to \infty} \rho_k = \frac{1-\mu}{1+\mu}.
\end{align*}
\end{theorem}
\begin{proof}
We distinguish two cases depending on the value of $L$. 

\textbf{Case 1:} $L\leq\frac{2}{2-\mu}-\mu.$ In this case, it holds that $$\rho_k\leq1-\mu\leq\frac{2}{L+\mu}-1.$$ 
Hence,  $\alpha_{k+1}=1+\rho_k\in(1,\frac{2}{L+\mu}]$ and therefore
\begin{align*} 
    \rho_{k+1}=\rho(I - \alpha_{k+1} A)= 1-\mu\alpha_{k+1}=\frac{1-\mu}{1+\mu}+\mu\left(\frac{1-\mu}{1+\mu}-\rho_k\right).
\end{align*}
Consequently, 
\begin{align*} 
    \left|\frac{1-\mu}{1+\mu}-\rho_{k+1}\right|=\mu\left|\frac{1-\mu}{1+\mu}-\rho_{k}\right|.
\end{align*}
Thus, the sequence $\rho_k$ converges to $\frac{1-\mu}{1+\mu}$ with contraction factor $\mu$.

\textbf{Case 2:} $\frac{2}{2-\mu}-\mu< L<1.$ For $k=1$, we get $\rho_1=\rho(I - A)=1-\mu$.  Whenever $\rho_k > \frac{2}{L+\mu}-1$, the update rule yields 
\begin{align*}
\rho_{k+1}=\rho(I - \alpha_{k+1} A)=L\alpha_{k+1}-1=L\rho_k+L-1 < \rho_k,
\end{align*}
so the sequence is strictly decreasing. Hence there exists an index $k_0\geq1$ such that  $\rho_{k_0} \geq \frac{2}{L+\mu}-1$ and $\rho_{k_0+1} < \frac{2}{L+\mu}-1$.  We now consider two subcases:

\textbf{Case 2.1:}  $\rho_{k_0+1}\geq\frac{1-\mu}{1+\mu}.$ Using the same relation as in Case 1, 
\begin{align*}
    \rho_{k_0+2}-\frac{1-\mu}{1+\mu}&=\mu\left(\frac{1-\mu}{1+\mu}-\rho_{k_0+1}\right)\\
    \rho_{k_0+3}-\frac{1-\mu}{1+\mu}&=\mu^2\left(\rho_{k_0+1}-\frac{1-\mu}{1+\mu}\right).
\end{align*}
This implies $\rho_{k_0+2}\leq\frac{1-\mu}{1+\mu}\leq\rho_{k_0+3}\leq\rho_{k_0+1}<\frac{2}{L+\mu}-1$. Hence, $\rho_{k_0+i}<\frac{2}{L+\mu}-1$ and $\alpha_{k_0+i}\in(1,\frac{2}{L+\mu})$  for $i\geq1$. Following the same procedure as in Case 1, the sequence $\{\rho_k\}$ converges geometrically to $\frac{1-\mu}{1+\mu}$. 

\textbf{Case 2.2:} $\rho_{k_0+1}<\frac{1-\mu}{1+\mu}.$ Since
\begin{align*}
    \rho_{k_0+1} = L \rho_{k_0} + L-1\geq L\left(\frac{2}{L+\mu}-1\right)+L-1=\frac{2L}{L+\mu} - 1,
\end{align*}
we obtain
\begin{align*}
    \rho_{k_0+2} - \frac{1-\mu}{1+\mu} = \mu\left(\frac{1-\mu}{1+\mu}-\rho_{k_0+1}\right)\leq \mu\left(\frac{2}{1+\mu} - \frac{2L}{L+\mu}\right) < \frac{2}{L+\mu}-1-\frac{1-\mu}{1+\mu},
\end{align*}
where the last inequality holds for $L<1$.
Hence $\frac{1-\mu}{1+\mu}<\rho_{k_0+2}<\frac{2}{L+\mu}-1$, which reduces to \textbf{Case 2.1}.  
 
In all cases, the sequence $\{\rho_k\}$ converges to $\frac{1-\mu}{1+\mu}$.
\end{proof}
\begin{remark}
\label{rk:gdL}
    By \Cref{thm:GD-rhok}, when $L<1$, we have
    \begin{align*}
        \lim\limits_{k \to \infty} \rho_k = \rho_{\texttt{GD}}^*+\delta_L,
    \end{align*}
    where $\delta_L:=\frac{1-\mu}{1+\mu}-\frac{L-\mu}{L+\mu}$ represents the discrepancy introduced by estimating $L$ as 1. As $L$ approaches 1, $\delta_L$ approaches 0. 
\end{remark}

\begin{remark}
 In the special case when $L = 1$, if $L \leq \frac{2}{2-\mu} - \mu$, then we have $\mu = 1$, and the adaptive algorithm converges in one iteration.  On the other hand, if $L > \frac{2}{2-\mu} - \mu$, it holds that
    \begin{align*}
    \rho_{k+1}=\alpha_{k+1}-1=\rho_k=\cdots=\rho_1=1-\mu.
    \end{align*}
    Hence, in the implementation with $\rho_k = \rho(I - \alpha_k A)$, if we observe $\rho_2 = \rho_1$, then it implies that $L = 1$ and $\mu = 1 - \rho_1$. We therefore can directly set $\alpha_k = \frac{2}{L + \mu} = \frac{2}{2-\rho_1}$ in \Cref{alg:AGD}.
\end{remark}

Computing the spectral radius $\rho_k = \rho(I-\alpha_k A)$ can be computationally expensive, especially for large-scale problems.  Therefore, in our implementation, we approximate $\rho_k$ using the geometric average of the ratios of successive residual norms
\begin{align}\label{average_GD}
\rho_k^{l} =
\begin{cases}
\left( \frac{\|\bm{r}_{k}\|}{\|\bm{r}_{0}\|} \right)^{\frac{1}{k}}, & k < l \\
\left(\prod_{i=k-l+1}^{k}\frac{\|\bm{r}_i\|}{\|\bm{r}_{i-1}\|}\right)^{\frac{1}{l}}, & k \geq l.
\end{cases}
\end{align}
If $l=1$, the above equation  simplifies to 
\begin{align*}
    \rho_{k}^1 = \frac{\|\bm{r}_{k}\|}{\|\bm{r}_{k-1}\|}.
\end{align*}
If $l=k$, it becomes
\begin{align*}
    \rho_{k}^k = \left( \frac{\|\bm{r}_{k}\|}{\|\bm{r}_{0}\|} \right)^{\frac{1}{k}}.
\end{align*}
Based on \eqref{average_GD}, we show that \Cref{alg:AGD} achieves a faster  convergence rate than GD with fixed step size $\alpha = 1$.
\begin{lemma}
    Let $\rho_k^l$ denote the geometric mean of the ratios of successive residual norms, as defined in \eqref{average_GD}. Then, the convergence rate of  \Cref{alg:AGD} satisfies 
    $$ 0\leq \frac{\|\bm{r}_{k}\|}{\|\bm{r}_{k-1}\|}\leq 1-\mu.$$
\end{lemma}
\begin{proof}
When $k=1$, with $\alpha_1=1$, we have $0\leq\rho_1^{l} = \frac{\|\bm{r}_{1}\|}{\|\bm{r}_{0}\|} \leq \|I-A\|=1-\mu$. 
    
\textbf{Case 1:} If $l=1$, assume $0\leq\rho_{k}^1 \leq 1-\mu$ for $k=2,\cdots$, then the step size at each iteration is given by 
$\alpha_{k + 1}=\frac{2}{1+\mu_{k}} = 1 + \rho_{k}.$ 
Thus, one arrives at $1\leq\alpha_{k+1}\leq2-\mu$ and  
\begin{align*}
     \rho_{k+1}^1 = \frac{\|\bm{r}_{k+1}\|}{\|\bm{r}_{k}\|}&\leq \|I-\alpha_{k+1}A\| = \max\{|1-\alpha_{k+1}\mu|, |1-\alpha_{k+1}L|\}\\
     &\leq \max\{1-\mu,(2-\mu)L-1\}=1-\mu.
\end{align*}

\textbf{Case 2:} For $1 < l < k$, suppose that $0 \le \rho_k^{l} \le 1-\mu$ holds for all $k = 2, \ldots, m$, where $m < l - 1$. Then, similarly, we obtain $1\leq\alpha_{k+1}\leq2-\mu$,  $\frac{\|\bm{r}_{k+1}\|}{\|\bm{r}_{k}\|}\leq 1-\mu$, and 
\begin{align*}
    \rho_{k+1}^{l} = \left(\frac{\|\bm{r}_{k+1}\|}{\|\bm{r}_{0}\|}\right)^{\frac{1}{k+1}}\leq 1-\mu. 
\end{align*}
This further gives $1\leq\alpha_{l}\leq 1-\mu$, $\frac{\|\bm{r}_{l}\|}{\|\bm{r}_{l-1}\|}\leq 1-\mu$, and $\rho_l^l\leq1-\mu$. By induction, one obtains that $\frac{\|\bm{r}_{k+1}\|}{\|\bm{r}_{k}\|}\leq 1-\mu$ and $\rho_{k+1}^{l}\leq 1-\mu$ for $k\geq l$. 

\textbf{Case 3:} If $l=k$, then by the same argument as in \textbf{Case 2}, we obtain the desired result. 
\end{proof}

The above lemma implies 
$$-\frac{L-\mu}{L+\mu}\leq\frac{\|\bm{r}_{k}\|}{\|\bm{r}_{k-1}\|}-\rho_{\texttt{GD}}^*\leq 1-\mu-\frac{L-\mu}{L+\mu},$$ which characterizes how the convergence rate of  \Cref{alg:AGD}  deviates from the optimal rate $\rho_{\texttt{GD}}^*$.  Next, we consider the special case $l=1$ in \eqref{average_GD}. This choice is computationally the most economical in practice and provides a clear insight into how the estimate $\rho_k$ relates to the optimal convergence rate $\rho_{\texttt{GD}}^*$.

\begin{theorem}\label{con_gd}
Let $\rho_k=\frac{\|\bm{r}_{k}\|}{\|\bm{r}_{k-1}\|}$. Suppose $\rho_k = \rho_{\texttt{GD}}^* + \delta_L+\epsilon_k,$  
where $\delta_L = \frac{1-\mu}{1+\mu}-\frac{L-\mu}{L+\mu}.$ Then, it holds that $\rho_{k+1} = \rho_{\texttt{GD}}^* + \delta_L+\epsilon_{k+1}$ and $\epsilon_{k+1}$ can be bounded as follows.

\textbf{Case 1:} If $L>\frac{2}{2-\mu}-\mu$, when $\epsilon_k\in(\frac{2}{L+\mu}-\frac{2}{1+\mu}, 1-\mu-\frac{1-\mu}{1+\mu}]$, 
\begin{align*}
    \epsilon_{k+1}\leq L\epsilon_k,
\end{align*}
and when $\epsilon_k \in(-\frac{1-\mu}{1+\mu},\frac{2}{L+\mu}-\frac{2}{1+\mu}]$,
\begin{align*}
\epsilon_{k+1}\leq  -\mu\epsilon_k.
\end{align*}

\textbf{Case 2:} If $L \leq \frac{2}{2-\mu}-\mu$, it becomes $\epsilon_k \in(-\frac{2-\mu}{1+\mu}, \frac{1}{L+\mu}-\frac{2}{1+\mu}]$, and
\begin{align*}
\epsilon_{k+1}\leq  -\mu\epsilon_k.
\end{align*}
\end{theorem}
\begin{proof}
For $k=1$, it yields $$\rho_1\leq \|I-\alpha_1A\|= 1-\mu=\rho_{\texttt{GD}}^*+\delta_L+\left(1-\mu-\frac{1-\mu}{1+\mu}\right).$$
Then, there exists some $\epsilon_1$ such that $\rho_1=\rho_{\texttt{GD}}^*+\delta_L+\epsilon_1$. Suppose $\rho_k=\rho_{\texttt{GD}}^*+\delta_L+\epsilon_k$. Observing that
\begin{align*}
    \rho_{k+1} = \frac{\|\bm{r}_{k+1}\|}{\|\bm{r}_{k}\|} \leq \|I-\alpha_{k+1}A\|= \max\{|1-\alpha_{k+1}\mu|, |1-\alpha_{k+1}L|\},
\end{align*}
we then estimate $\rho_{k+1}$ in different cases:

\textbf{Case 1:} If $L> \frac{2}{2-\mu}-\mu$, direct computation gives $\epsilon_k \in (-\frac{1-\mu}{1+\mu}, 1-\mu-\frac{1-\mu}{1+\mu}]$.  We further consider two different cases.

\textbf{Case 1.1}: If $\epsilon_k \in(\frac{2}{L+\mu}-\frac{2}{1+\mu}, 1-\mu-\frac{1-\mu}{1+\mu}]$, it follows that  $\alpha_{k+1} \in \left(\frac{2}{L + \mu}, 2-\mu \right]$. Also, we have the following estimates:
\begin{align*}
\rho_{k+1}&\leq L(1+\rho_k)-1 = L(1 + \rho_{\texttt{GD}}^* +\delta_L+ \epsilon_{k}) - 1 \\
&= \rho_{\texttt{GD}}^* + L\delta_L + L\epsilon_k  - (1-L)(\rho_{\texttt{GD}}^*+1)\leq \rho_{\texttt{GD}}^* + \delta_L + L\epsilon_k.
\end{align*}
This gives $ \epsilon_{k+1}\leq L\epsilon_{k}.$ 

\textbf{Case 1.2}: If $\epsilon_k\in(-\frac{1-\mu}{1+\mu},\frac{2}{L+\mu}-\frac{2}{1+\mu}]$, we obtain  $\alpha_{k+1}\in(1,\frac{2}{L+\mu}]$ and
\begin{align*}
    \rho_{k+1} &\leq 1-\mu(1+\rho_k) = 1-\mu(1+\rho_{\texttt{GD}}^*+\delta_L+\epsilon_k )\\
    & =\rho_{\texttt{GD}}^*-\mu\delta_L -\mu\epsilon_k - (1+\mu)\left(\rho_{\texttt{GD}}^* - \frac{1-\mu}{1+\mu}\right) = \rho_{\texttt{GD}}^* + \delta_L-\mu\epsilon_k .
\end{align*}
Hence, we have $\epsilon_{k+1}\leq -\mu\epsilon_k$.

\textbf{Case 2:} If $L\leq\frac{2}{2-\mu}-\mu$, we have $\epsilon_k +\delta_L\in(-\frac{L-\mu}{L+\mu},\frac{2(1-L)}{L+\mu}]$ and $\alpha_{k+1}\in(1,\frac{2}{L+\mu}]$. Following the similar analysis as in \textbf{Case 1.2} above, we obtain that
 $\rho_{k+1}\leq \rho_{\texttt{GD}}^* + \delta_L-\mu\epsilon_k$ and, thus,  $\epsilon_{k+1}\leq -\mu\epsilon_k$.
 
 Combining all the above cases, we complete the proof.
\end{proof}

\begin{remark}\label{remrk_gd} 
\Cref{con_gd} shows that the deviation $\rho_k - \rho_{\texttt{GD}}^*$ consists of two  contributions:  $\epsilon_k$, which accounts for the error introduced by the adaptive update, and $\delta_L$, which results from the rescaling of $A$.
    
   When the largest eigenvalue is accurately estimated, the rescaled matrix satisfies $L\approx 1$ and, in particular, $L \ge \frac{2}{2 - \mu} - \mu$. In this regime, the upper bound of the sequence $\{\rho_k\}$ decreases as described in \textbf{Case 1}. Conversely, if the largest eigenvalue is not estimated with sufficient accuracy, the problem falls into the regime of \textbf{Case 2}. In this situation, no lower bound is available for $\epsilon_{k+1}$, and the convergence rate may fluctuate in practice. 
The numerical experiments in Section~\ref{subsect:quadopt} further illustrate that, when the largest eigenvalue is not estimated accurately enough, the convergence rate of $\rho_k$  still approaches $\frac{1-\mu}{1+\mu}$, but exhibits a noticeable gap with $\rho_{\texttt{GD}}^*$ due to $\delta_L$.

\end{remark}


\subsection{Adaptive Accelerated Gradient Descent}
In this section, we discuss how to extend the idea of adaptive GD to accelerated methods, such as NAG and HB methods.

\subsubsection{Adaptive Nesterov Acceleration}
Based on the update rule for NAG method, we have
\begin{align*}
\boldsymbol{x}^{k+1} & =\boldsymbol{x}^k-\alpha \nabla f\left(\boldsymbol{x}^k+\beta\left(\boldsymbol{x}^k-\boldsymbol{x}^{k-1}\right)\right)+\beta\left(\boldsymbol{x}^k-\boldsymbol{x}^{k-1}\right)\\
& = \boldsymbol{x}^k-\alpha \left( A\left(\boldsymbol{x}^k+\beta\left(\boldsymbol{x}^k-\boldsymbol{x}^{k-1}\right)\right)-\boldsymbol{b}\right)+\beta\left(\boldsymbol{x}^k-\boldsymbol{x}^{k-1}\right)
\end{align*}
Subtracting $\boldsymbol{x}^*$ from both sides gives
\begin{align*}
\boldsymbol{x}^{k+1}-\boldsymbol{x}^{*}=(1+\beta)(I-\alpha A)(\boldsymbol{x}^k-\boldsymbol{x}^{*})-\beta(I-\alpha A)(\boldsymbol{x}^{k-1}-\boldsymbol{x}^{*}).
\end{align*}
Define 
\begin{align*}
M := 
\begin{pmatrix}
  (1+\beta)(I-\alpha A) & -\beta(I-\alpha A) \\
  I & 0
\end{pmatrix}\quad\text{and}\quad \mathcal{R}_k:=\binom{\bm{r}_{k}}{\bm{r}_{k-1}}.
\end{align*}
Then the residuals for the NAG follow a two-term recurrence:
\begin{align}\label{resi_NAG}
  \mathcal{R}_{k+1}=M\mathcal{R}_{k}, \quad\text{for}\quad k \geq 1.
\end{align}
Note that with the optimal parameters given in \Cref{tab:classic_methods}, we get $\rho_{\texttt{NAG}}^* = 1-\frac{\sqrt{\mu}}{\sqrt{L}}$. Solving for the strong convexity parameter yields $\mu = L\left(1-\rho_{\texttt{NAG}}^*\right)^2$. 

Our adaptive NAG method follows the same procedure as NAG but employs different parameters, i.e., 
\begin{align*}
    \bm{y}_{k} &= \bm{x}_{k} + \beta_{k+1} (\bm{x}_{k} - \bm{x}_{k-1})\\
    \bm{x}_{k+1} &= \bm{y}_k - \alpha_{k+1} (A \bm{y}_k - \bm{b}),
\end{align*}
where the step size $\alpha_{k}$ and momentum parameter $\beta_{k}$ are determined adaptively based on information from the previous iterations. Consequently, the residual of the adaptive NAG satisfies a similar recurrence relation, $\mathcal{R}_{k+1}=M_{k} \mathcal{R}_{k}$, where the error propagation matrix $M_k$ is given by
\begin{align*}
    M_k=\begin{pmatrix}
(1+\beta_k)(I-\alpha_k A)& -\beta_k(I-\alpha_k A) \\
I & 0
\end{pmatrix}.
\end{align*}
Let $\rho_k$ be an empirical approximation of $\rho_{\texttt{NAG}}^*$. Following the same idea of adaptive GD, since $\mu = L\left(1-\rho_{\texttt{NAG}}^*\right)^2$, we again use $L=1$ and approximate $\mu$ at iteration $k$ as $\mu_{k} = \left(1-\rho_{k}\right)^2$. Consequently, $\alpha_{k+1}$ and $\beta_{k+1}$ are updated as follows: 
\begin{align*}
\alpha_{k+1}= 1 \quad\text{and}\quad
\beta_{k+1} = \frac{1-\sqrt{\mu_{k}}}{1+\sqrt{\mu_{k}}}=\frac{\rho_k}{2-\rho_k}. 
\end{align*}
The adaptive NAG is summarized in \Cref{alg:adaptive_nag}

\begin{algorithm}[h!]
\caption{Adaptive Nesterov Accelerated Gradient}
\begin{algorithmic}[1]\label{alg:adaptive_nag}
\STATE \textbf{Input:} SPD matrix $A \in \mathbb{R}^{n \times n}$, vector $\bm{b} \in \mathbb{R}^n$, initial point $\bm{x}_0 \in \mathbb{R}^n$
\STATE \textbf{Initialize:} Compute $\bm{r}_0 = \bm{b} - A \bm{x}_0$, 
\STATE Perform first step using GD: $\alpha_1 = 1$, \quad $\bm{x}_1 = \bm{x}_0 - \alpha_1 \nabla f(\bm{x}_0)$
\STATE $\bm{r}_1 = \bm{b} - A \bm{x}_1$, \quad $\rho_1 = \|\bm{r}_1\| / \|\bm{r}_0\|$
\FOR{$k = 2, 3, \dots$}
    \STATE $\alpha_{k} = 1$
    \STATE $\beta_{k} = \frac{\rho_{k-1}}{2-\rho_{k-1}}$
    \STATE $\bm{y}_{k-1} = \bm{x}_{k-1} + \beta_{k} (\bm{x}_{k-1} - \bm{x}_{k-2})$
    \STATE $\bm{x}_{k} = \bm{y}_{k-1} - \alpha_k \nabla f(\bm{y}_{k-1})$
    \STATE $\rho_k \approx \rho_{\texttt{NAG}}^*,\text{ where } \rho_k \text{ is an estimate of } \rho_{\texttt{NAG}}^*$
\ENDFOR
\end{algorithmic}
\end{algorithm}

We first present a lemma showing that, if the estimation $\rho_k$ is obtained by computing the spectral radius of $M_k$, then the proposed adaptive NAG \Cref{alg:adaptive_nag} always converges faster than GD with $\alpha = 1$.
\begin{lemma}\label{NAG_lem}
    Let $\rho_k= \rho(M_k)$ in  \Cref{alg:adaptive_nag}. Then, we have
$$0\leq\rho_k \leq 1-\mu.$$
\end{lemma}
\begin{proof}
Given that the solution at the first step is updated using the GD with $\alpha_0 =1$, it naturally yields $\rho_1 \leq 1 - \mu$. We now turn to the case when $k \geq 2$. 
Since $A$ is SPD, it admits an eigenvalue decomposition $A = U \Lambda U^\top$,
where $U$ is orthogonal and  $\Lambda=\operatorname{diag}(\lambda_1,\cdots,\lambda_n)$. By applying a suitable reordering of the coordinates, $M_k \in \mathbb{R}^{2n\times 2n}$ can be transformed into a block-diagonal form with $2\times 2$ diagonal blocks, i.e.,
\begin{align*}
B_k = \operatorname{diag}(B_k^1, B_k^2, \cdots, B_k^n), \ \text{where }	
B_k^i=\begin{pmatrix} (1+\beta_k) (1- \lambda_i) & -\beta_k (1-\lambda_i) \\ 1 & 0 \end{pmatrix}.
\end{align*}
Therefore, $\rho_k$ equals the maximum spectral radius of $B_k^i$ over all $i$. 
The eigenvalues of $B_k^i$ can be obtained by solving its characteristic polynomial:
\begin{align}\label{chara_NAG} 
\theta^2 - (1 + \beta_k)( 1- \lambda_i)\theta + \beta_k (1-\lambda_i)
=0.
\end{align}
When $\beta_k \geq \frac{1-\sqrt{\lambda_i}}{1+\sqrt{\lambda_i}}$, the above characteristic equation has either a repeated real root or a pair of complex conjugate roots, with magnitude $\sqrt{\beta_k(1-\lambda_i)}$. 
Recall that  
\begin{align*}
    \beta_{k+1}= \frac{\rho_{k}}{2-\rho_{k}},\quad\mu_k = (1-\rho_k)^2.
\end{align*}
We now analyze the behavior of $\beta_{k+1}$  under different ranges of $\rho_{k}$.

\textbf{Case 1}: When $\rho_{k}\in [1-\sqrt{\mu},1-\mu]$, then  $\beta_{k+1}\geq\frac{1-\sqrt{\mu}}{1+\sqrt{\mu}}$ and
\begin{align*}   \rho_{k+1}=\sqrt{\beta_{k+1}(1-\mu)}=\sqrt{\frac{\rho_{k}}{2-\rho_{k}}(1-\mu)}\leq \frac{1-\mu}{\sqrt{1+\mu}}<1-\mu.
\end{align*} 

\textbf{Case 2}: When $\rho_{k}\in\left(0,1-\sqrt{\mu}\right)$, there exist blocks whose characteristic equation \eqref{chara_NAG} admits two distinct real roots for $\lambda_i\neq 1$, where the larger one is given by:
\begin{align*}    \theta(\lambda_i,\beta_{k+1}) = \frac{(1+\beta_{k+1})(1-\lambda_i)+\sqrt{(1+\beta_{k+1})^2(1-\lambda_i)^2-4\beta_{k+1}(1-\lambda_i)}}{2}.
\end{align*}
Since the discriminant 
$(1 + \beta_{k+1})^2 (1 - \lambda_i)^2 - 4\beta_{k+1} (1 - \lambda_i) > 0$ for $\lambda_i\neq1$,
it follows that 
$(1 + \beta_{k+1})^2 (1 - \lambda_i) > 4\beta_{k+1} > 2\beta_{k+1}$. Differentiating $\theta$ with respect to $\lambda_i$, we obtain
\begin{align*}
    \frac{d\theta}{d\lambda_i} = -\frac{1+\beta_{k+1}}{2} + \frac{-(1+\beta_{k+1})^2(1-\lambda_i)+2\beta_{k+1}}{2\sqrt{(1+\beta_{k+1})^2(1-\lambda_i)^2-4\beta_{k+1}(1-\lambda_i)}}<0.
\end{align*}
Hence, $\theta(\lambda_i,\beta_{k+1})$ is a decreasing function with respect to $\lambda_i$. Similarly, differentiating $\theta$ with respect to $\beta_{k+1}$, we obtain
\begin{align*}
    \frac{d\theta}{d\beta_{k+1}}&=\frac{1-\lambda_i}{2}+\frac{(1-\lambda_i)^2(1+\beta_{k+1})-2(1-\lambda_i)}{2\sqrt{(1+\beta_{k+1})^2(1-\lambda_i)^2-4\beta_{k+1}(1-\lambda_i)}}\\
    &=\frac{(1-\lambda_i)(\sqrt{(1+\beta_{k+1})^2(1-\lambda_i)^2-4\beta_{k+1}(1-\lambda_i)}+(1-\lambda_i)(1+\beta_{k+1})-2)}{2\sqrt{(1+\beta_{k+1})^2(1-\lambda_i)^2-4\beta_{k+1}(1-\lambda_i)}}.
\end{align*}
A straightforward computation shows that
\begin{align*}
    &\sqrt{(1+\beta_{k+1})^2(1-\lambda_i)^2-4\beta_{k+1}(1-\lambda_i)}+(1-\lambda_i)(1+\beta_{k+1})-2\\
    &=\sqrt{(1+\beta_{k+1})^2(1-\lambda_i)^2-4(1+\beta_{k+1})(1-\lambda_i)+4(1-\lambda_i)}+(1-\lambda_i)(1+\beta_{k+1})-2\\
    &<\sqrt{(1+\beta_{k+1})^2(1-\lambda_i)^2-4(1+\beta_{k+1})(1-\lambda_i)+4}+(1-\lambda_i)(1+\beta_{k+1})-2\\
    &=0.    
\end{align*}
This leads to $\frac{d\theta}{d\beta_{k+1}}<0$. Since $\theta(\lambda_i,\beta_{k+1})$ is strictly decreasing with respect to both $\lambda_i$ and $\beta_{k+1}$ for all $\lambda_i<1$, we can bound $\rho_{k+1}$ as
\begin{align*}
    \rho_{k+1}\leq \max\{\theta(\lambda_i,\beta_{k+1}),\theta(1,\beta_{k+1})\}<\max\{\theta(\mu,0),0\}=1-\mu.
\end{align*}
This completes the proof.
\end{proof}

Next, we further show that $\rho_k$ converges to $1 - \sqrt{\mu}$, which is the optimal convergence rate of standard NAG with $L=1$, demonstrating \Cref{alg:adaptive_nag} can achieve acceleration.
\begin{theorem}\label{Thm:ada_NAG}
Suppose that $\rho_k= \rho(M_k)$ in  \Cref{alg:adaptive_nag}. It follows that
$$\lim\limits_{k \to \infty} \rho_k = 1 - \sqrt{\mu}.$$
\end{theorem}

\begin{proof}
For $k=1$, it holds that $0\leq\rho_1\leq1-\mu$.  

We first consider the case when $\rho_1\in[1-\sqrt{\mu},1-\mu]$. In this case, we have $\beta_{2}\geq\frac{1-\sqrt{\mu}}{1+\sqrt{\mu}}$. A brief analysis reveals that $$\rho_{k+1}=\sqrt{\beta_k(1-\mu)}=\sqrt{\frac{\rho_{k}}{2-\rho_{k}}(1-\mu)}$$ 
is an increasing function with respect to $\rho_k$ on the interval $(0, 1]$. This implies that if $\rho_1 \leq 1-\mu$, then $\rho_2 \leq \frac{1-\mu}{\sqrt{1+\mu}}$, and more generally, $\rho_k \in [1-\sqrt{\mu}, 1-\mu]$ for all $k \geq 2$.  
We can further show the sequence $\{\rho_k\}$ is nonincreasing: $\rho_{k+1}\leq \rho_k.$ By contradiction, suppose $\rho_{k+1}>\rho_k$. Then, we have \[\rho_{k+1}^2 = \frac{\rho_{k}}{2-\rho_{k}}(1-\mu)>\rho_k^2 \Longrightarrow \frac{1-\mu}{2-\rho_k} > \rho_k.\] Rearranging terms leads to $\mu < (1-\rho_k)^2$, which directly contradicts with $\rho_k\geq 1-\sqrt{\mu}$.  Consequently, the sequence $\{\rho_k\}$ is monotonically decreasing and bounded below. By the Monotone Convergence Theorem, the sequence converges to a finite limit $\hat{\rho}$. At the limit, the fixed-point equation must hold:$$\hat{\rho} = \sqrt{\frac{\hat{\rho}}{2-\hat{\rho}}(1-\mu)}.$$ Solving for $\hat{\rho}$ yields  $\lim\limits_{k \to \infty} \rho_k = \hat{\rho} = 1 - \sqrt{\mu}$. 

On the other hand, when  $\rho_1\in(0,1-\sqrt{\mu})$, then by invoking Lemma \ref{NAG_lem}, one finds that $\frac{d\theta}{d\rho_k}=\frac{d\theta}{d\beta_k}\frac{d\beta_k}{d\rho_k}<0$, which implies that $\rho_2\in(1-\sqrt{\mu},1-\mu)$. A similar argument as the case above yields $\rho_k\in[1-\sqrt{\mu},1-\mu]$ for $k\geq2$. Consequently, $\rho_k$ converges to $1-\sqrt{\mu}$, which completes the proof.
\end{proof}
\begin{remark}
\Cref{Thm:ada_NAG} establishes that $\rho_k$ approaches the convergence factor $1-\sqrt{\mu}$ corresponding to NAG with $\alpha = 1$ and $\beta=\frac{1-\sqrt{\mu}}{1+\sqrt{\mu}}$. Similar to the GD case, we can rewrite it as
    \begin{align*}
    \lim\limits_{k \to \infty} \rho_k = 1- \sqrt{\mu}=\rho_{\texttt{NAG}}^*+\left(\frac{\sqrt{\mu}}{\sqrt{L}} - \sqrt{\mu}\right),
\end{align*}
where the second term $\left(\frac{\sqrt{\mu}}{\sqrt{L}} - \sqrt{\mu}\right)$ represents the discrepancy introduced by estimating $L$ as $1$ (cf. Remark~\ref{rk:gdL}).

We now consider the case $L = 1$. Assume $\rho_k:=\rho_{\texttt{NAG}}^*+\epsilon_k.$ When, $\epsilon_{k}\in [0,\sqrt{\mu}-\mu]$, then  $\beta_{k+1}\geq\frac{1-\sqrt{\mu}}{1+\sqrt{\mu}}$.  With $\mu = (1-\rho_{\texttt{NAG}}^*)^2$, we further have 
\begin{align*}
    \rho_{k+1}
    &=\sqrt{\beta_{k+1}(1-\mu)}\\
    &=\sqrt{\left(\frac{\rho_{\texttt{NAG}}^* + \epsilon_{k}}{2-\rho_{\texttt{NAG}}^*-\epsilon_{k}}\right)\left(1-(1-\rho_{\texttt{NAG}}^*)^2\right)}\\
    &=\sqrt{\left(\rho_{\texttt{NAG}}^* + \frac{\epsilon_{k}}{2-\rho_{\texttt{NAG}}^*-\epsilon_{k}}\right)^2 -\left(\frac{\epsilon_{k}}{2-\rho_{\texttt{NAG}}^*-\epsilon_{k}}\right)^2}\\
    &\leq \rho_{\texttt{NAG}}^* + \frac{\epsilon_{k}}{2-\rho_{\texttt{NAG}}^*-\epsilon_{k}}\\
    &\leq \rho_{\texttt{NAG}}^*+\frac{1}{1+\mu}\epsilon_{k}.
\end{align*}
It follows that $\epsilon_{k+1}\leq \frac{1}{1+\mu}\epsilon_{k}$. Then, $\epsilon_k$ is a monotone and bounded sequence that converges to 0 at a linear rate with factor $\frac{1}{1+\mu}$. This implies that the upper bound of $\rho_{k+1}$ approaches to $\rho_{\texttt{NAG}}^*$.
\end{remark}


As in the adaptive GD case, for practical computation, we approximate $\rho_{\texttt{NAG}}^{*}$ by $\rho_k$ computed from the geometric average of the ratios of successive residual norms
\begin{align}\label{average_NAG}
\rho_k^{l} =
\begin{cases}
\left( \frac{\|\mathcal{R}_{k}\|}{\|\mathcal{R}_{0}\|} \right)^{\frac{1}{k}}, & k < l \\
\left(\prod_{i=k-l+1}^{k}\frac{\|\mathcal{R}_i\|}{\|\mathcal{R}_{i-1}\|}\right)^{\frac{1}{l}}, & k \geq l.
\end{cases}
\end{align}
For $k\geq l$, we obtain
\begin{align*}
    \rho_{k}^{l} \leq \left(\prod_{i=k-l+1}^{k}\|M_i\|\right)^{\frac{1}{l}}.
\end{align*}
For a general matrix $M$, the spectral radius is always bounded above by the matrix 2-norm, i.e., $\rho(M) \leq \|M\|$. A sharp characterization of the convergence behavior via the matrix 2-norm is nontrivial, and we defer a rigorous treatment  for future work.  Nonetheless, the numerical results presented in Section \ref{numeri_test}  support the effectiveness of Algorithm \ref{alg:adaptive_nag} when employing the approximated $\rho_k$ as defined in \eqref{average_NAG}. 
\begin{remark}
    The proposed framework offers broader insights for construction of adaptive variants for a variety of momentum-based algorithms, including the Heavy-Ball (HB) method, HNAG+ \cite{chen2025hnag++}, the Accelerated Over-Relaxation HB method \cite{wei2025acceleratedoverrelaxationheavyballmethod}, the Triple Momentum Method \cite{van2017fastest}, $C^2$-Momentum \cite{van2025fastest}, and the Information-Theoretic Exact Method \cite{taylor2023optimal}). In these methods, the optimal parameters can be expressed in terms of the convergence factor. To illustrate this adaptability, we present an  adaptive variant of the HB method, detailed in \Cref{alg:adaptive_hb}. For practical implementation, we estimate the optimal convergence rate $\rho_{\texttt{HB}}^*$ by computing the geometric average of the ratios of successive residual norms, as described in \eqref{average_NAG}. 
\end{remark}
\begin{algorithm}[h!]
\caption{Adaptive Heavy Ball Method}
\begin{algorithmic}[1]\label{alg:adaptive_hb}
\STATE \textbf{Input:} SPD matrix $A \in \mathbb{R}^{n \times n}$, vector $\bm{b} \in \mathbb{R}^n$, initial point $\bm{x}_0 \in \mathbb{R}^n$
\STATE \textbf{Initialize:} Compute $\bm{r}_0 = \bm{b} - A \bm{x}_0 $
\STATE Update the first step using GD: $\alpha_1 = 1$, \quad $\bm{x}_1 = \bm{x}_0 - \alpha_1 \nabla f(\bm{x}_0)$
\STATE $\bm{r}_1 = A \bm{x}_1 - \bm{b}$, \quad $\rho_1 = \|\bm{r}_1\| / \|\bm{r}_0\|$
\FOR{$k = 2, 3, \dots$}
    \STATE $\alpha_{k}= (1+\rho_{k-1})^2$
    \STATE $\beta_{k} = \rho_{k-1}^2$
    \STATE $\bm{x}_{k} = \bm{x}_{k-1} - \alpha_k \nabla f(\bm{x}_{k-1}) + \beta_k (\bm{x}_{k-1} - \bm{x}_{k-2})$
    \STATE $\rho_k \approx \rho_{\texttt{HB}}^*,\text{ where } \rho_k \text{ is an estimate of } \rho_{\texttt{HB}}^*$
\ENDFOR
\end{algorithmic}
\end{algorithm}

\section{Numerical Experiments}\label{numeri_test}
In this section, we present numerical results to evaluate the performance of the proposed  adaptive GD, NAG, and HB methods. We compare these adaptive variants against their standard counterparts configured with theoretically optimal parameters. Notably, various approaches can be used to approximate $\rho^*$, with the main idea being to approximate the spectral radius of the underlying iteration matrix using observable and computationally inexpensive quantities. In the following experiments, we consider the adaptive strategies described in \eqref{average_GD} and \eqref{average_NAG} to update parameters dynamically. 

\subsection{Quadratic optimization problem}
\label{subsect:quadopt}
We consider solving quadratic optimization problems with diagonal matrix $A$ that have varying eigenvalue distributions. First, to examine the behavior described in \Cref{con_gd}, we construct diagonal matrices whose eigenvalues are randomly distributed.  
We sample $\mu \sim \mathcal{U}(0.2,\,0.4)$, 
where $\mathcal{U}(a, b)$ denotes the uniform distribution over the interval $[a , b]$. 
We consider two cases:
\begin{align*}
    L_1 = \frac12\!\left(1 + \frac{2}{2-\mu} - \mu\right), \qquad
    L_2 = \frac12\!\left(\mu + \frac{2}{2-\mu} - \mu\right).
\end{align*}
By construction, $L_1 > \frac{2}{2-\mu}$, whereas $L_2$ fails to meet this condition, corresponding to \textbf{Cases 1} and \textbf{2} in \cref{con_gd}, respectively. Given $\mu$ and $L = L_j$, $j=1, \, 2$, the eigenvalues are generated uniformly in $[\mu,\,L]$ with the smallest and largest fixed, i.e.,
\begin{align*}
\lambda_1 = \mu, \quad 
\lambda_n = L_j, \quad
\lambda_i \in \mu + (L_j-\mu)\,\mathcal{U}(0,1),\quad 2 \le i \le n-1, \quad j = 1, \, 2.
\end{align*}
\begin{figure}[htbp]
    \centering
    \begin{minipage}[t]{0.49\textwidth}
        \centering
        \includegraphics[width=\textwidth]{ 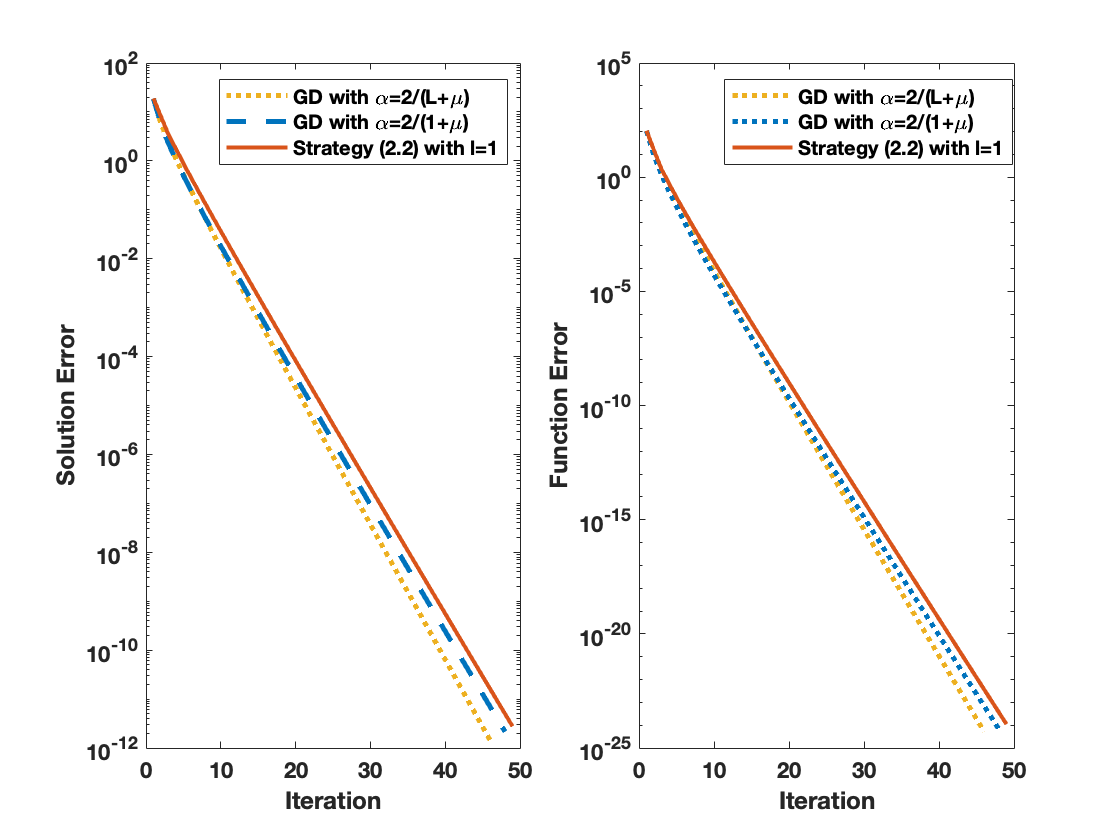}
    \end{minipage}%
    \hfill
    \begin{minipage}[t]{0.49\textwidth}
        \centering
        \includegraphics[width=\textwidth]{ 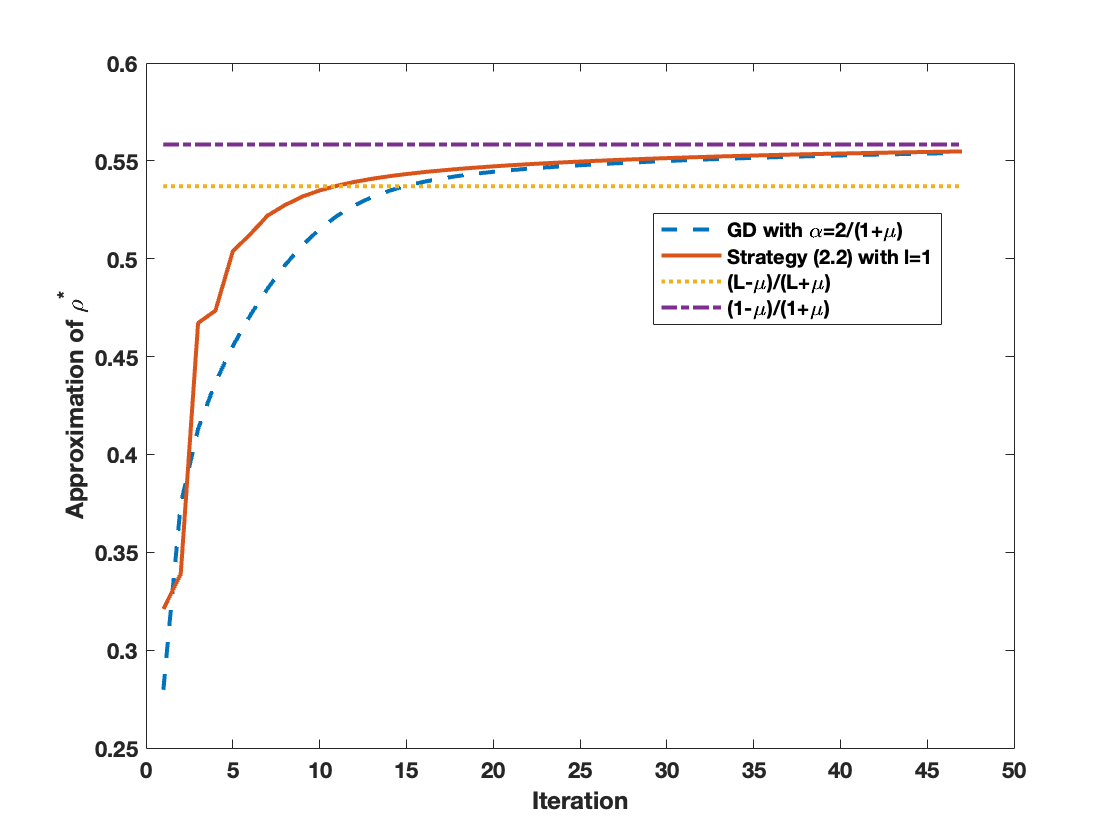}
    \end{minipage}
    \caption{Error (left) and estimated $\rho^*$ (right) for GD on a diagonal matrix with random eigenvalue distribution satisfying $L>\frac{2}{2-\mu}-\mu$.}\label{gd_condition}
\end{figure}
\begin{figure}[htbp]
    \centering
    \begin{minipage}[t]{0.49\textwidth}
        \centering
        \includegraphics[width=\textwidth]{ 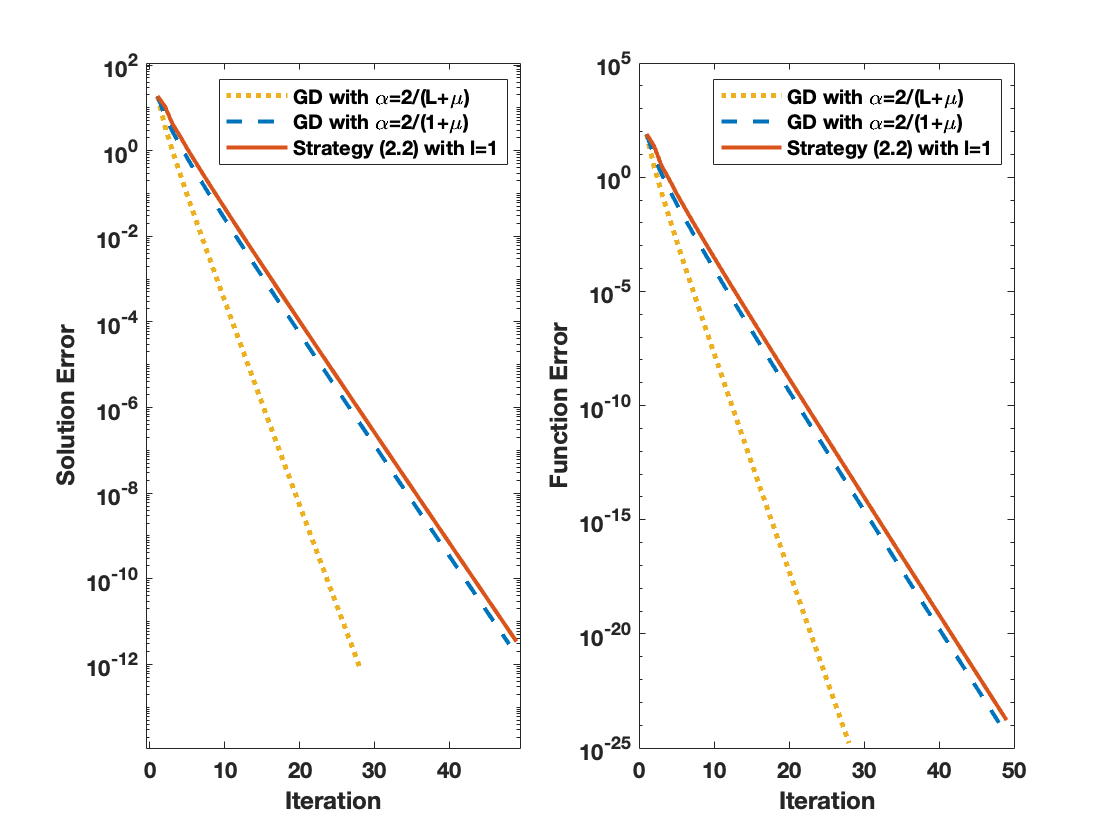}
    \end{minipage}%
    \hfill
    \begin{minipage}[t]{0.49\textwidth}
        \centering
        \includegraphics[width=\textwidth]{ 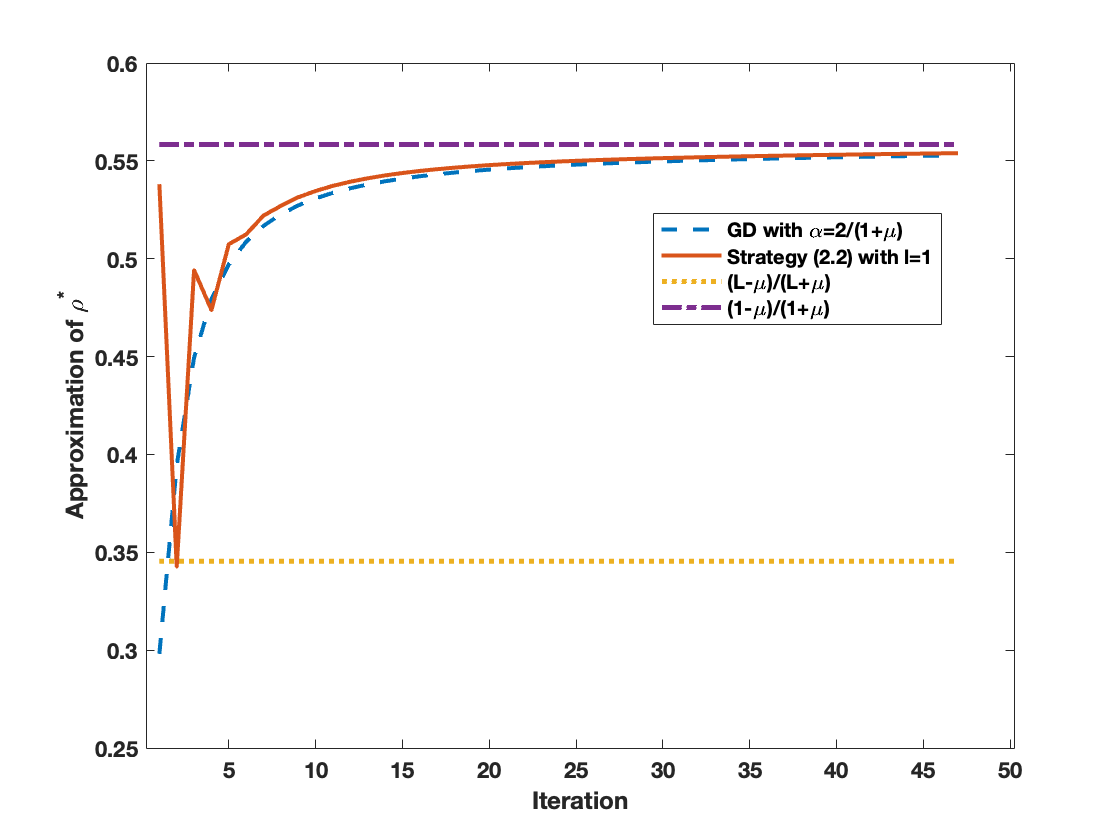}
    \end{minipage}
    \caption{Error (left) and estimated $\rho^*$ (right) for GD on a diagonal matrix with random eigenvalue distribution satisfying $L\leq\frac{2}{2-\mu}-\mu$.}\label{gd_condition_no}
\end{figure}
Let $n=1000$. The vector $\bm{b}$ is set to the zero vector. As a result, the quadratic optimization problem $f$ attains its minimum at $\bm{x}^* = \bm{0}$. All methods are initialized with a random vector $\bm{x}_0 \in \mathbb{R}^n$ generated by \texttt{rand}$(n, 1)$.
The iteration is terminated when either iteration count reaches a maximum \texttt{maxIt}, or the gradient norm $\|\nabla f(\bm{x}_k)\|$ is less than a tolerance \texttt{tol}.
Unless otherwise specified, we use \texttt{maxIt} = 3000 and \texttt{tol} = $10^{-12}$. In addition, we define the solution error as $\|\bm{x}_k - \bm{x}^*\|$, and the function error as $f(\bm{x}_k) - f(\bm{x}^*)$.

The behavior of the error and approximation of $\rho_{\texttt{GD}}^*$ across the iterations $k$ for two cases are shown in \cref{gd_condition} and \ref{gd_condition_no}. For comparison, we also plot the results of GD with $\alpha = \frac{1-\mu}{1+\mu}$, where the evolution of $\rho_k$ is characterized by $\|\bm r_k\| / \|\bm r_{k-1}\|$.
The results are consistent with the observation made in \Cref{remrk_gd}. In both cases, $\rho_k$ converges to $\frac{1-\mu}{1+\mu}$, which is the convergence factor of GD with step size $\alpha = \frac{2}{1+\mu}$, rather than the optimal rate $\rho_{\texttt{GD}}^{*}$. In particular, larger deviations of $L$ from 1, corresponding to poor rescaling of the Hessian in general cases, result in a convergence factor that is increasingly farther from the optimal one.

Next, we study three methods employing adaptive strategies described in \eqref{average_GD} and \eqref{average_NAG}, with $l = 1, 5, \text{ and } k$, respectively. The diagonal Hessians are constructed as follows.

\emph{Uniform}: the eigenvalues are evenly spaced in the interval $[1, 1000]$; that is,
$$\lambda_i \in \texttt{linspace}(1, 1000, n).$$

\emph{Log}: the eigenvalues are logarithmically spaced from 1 to $10^5$, representing a wide dynamic range; that is,
$$\lambda_i \in \texttt{logspace}(0, 5, n).$$

\emph{Cluster}: 90\% of the eigenvalues are sampled uniformly in the narrow range $[0, 0.1]$, while the remaining 10\% are outliers centered around 0.7; that is,
\begin{align*}
\lambda_\text{cluster} \sim \mathcal{U}(0,0.1), \quad \lambda_\text{outlier} \sim \mathcal{U}(0.65,0.75), \quad\text{then}
\quad\lambda_i \in \{\lambda_\text{cluster}, \lambda_\text{outlier}\}.
\end{align*}
Assuming that a more accurate estimate of the largest eigenvalue $\lambda_n$ is available, we can rescale the matrix $A$ such that $L = 1$.  Here, for standard methods, we use strategy \eqref{average_GD} and \eqref{average_NAG} with $l=1$ to estimate $\rho_k$ at each iteration. Besides, for NAG and HB, we initialize the first iteration using a single GD step with fixed step size $\alpha_0 = 1$: $ \bm{x}_{1} = \bm{x}_{0} - \alpha_0 \nabla f(\bm{x}_{0}).$ 


\begin{figure}[htbp]
    \centering
    \begin{minipage}[t]{0.49\textwidth}
        \centering
        \includegraphics[width=\textwidth]{ 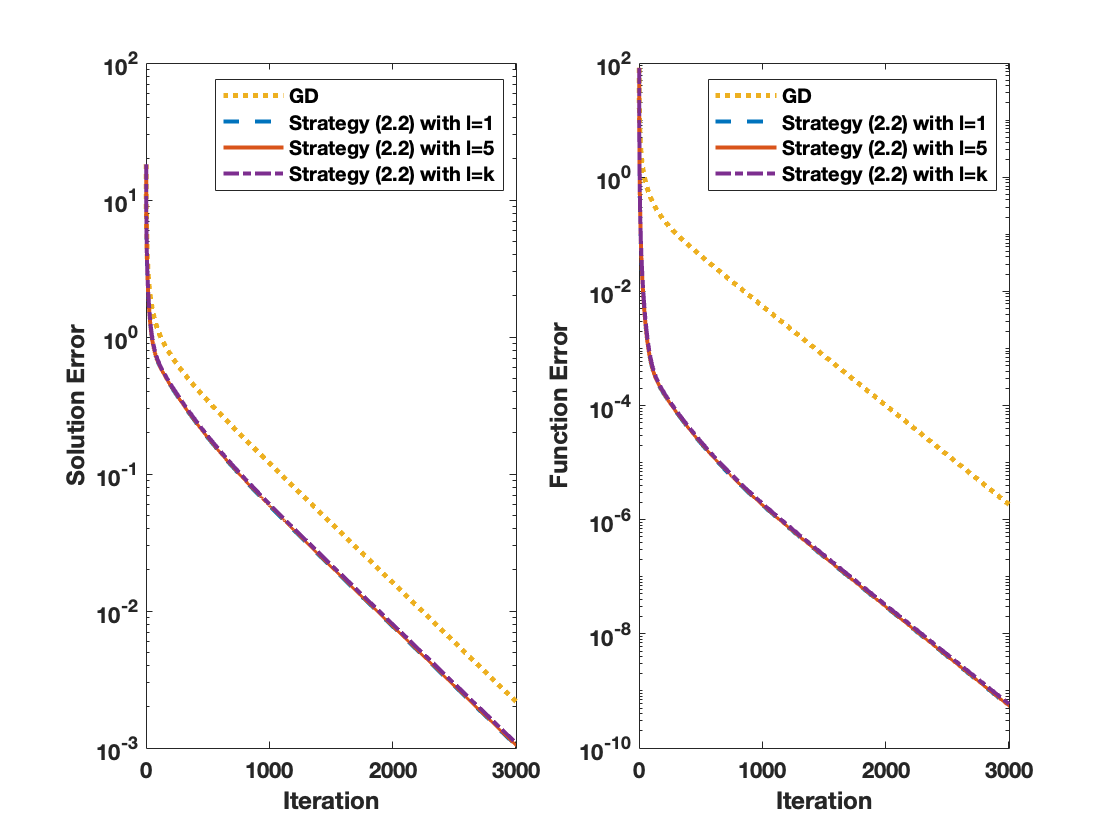}
    \end{minipage}%
    \hfill
    \begin{minipage}[t]{0.49\textwidth}
        \centering
        \includegraphics[width=\textwidth]{ 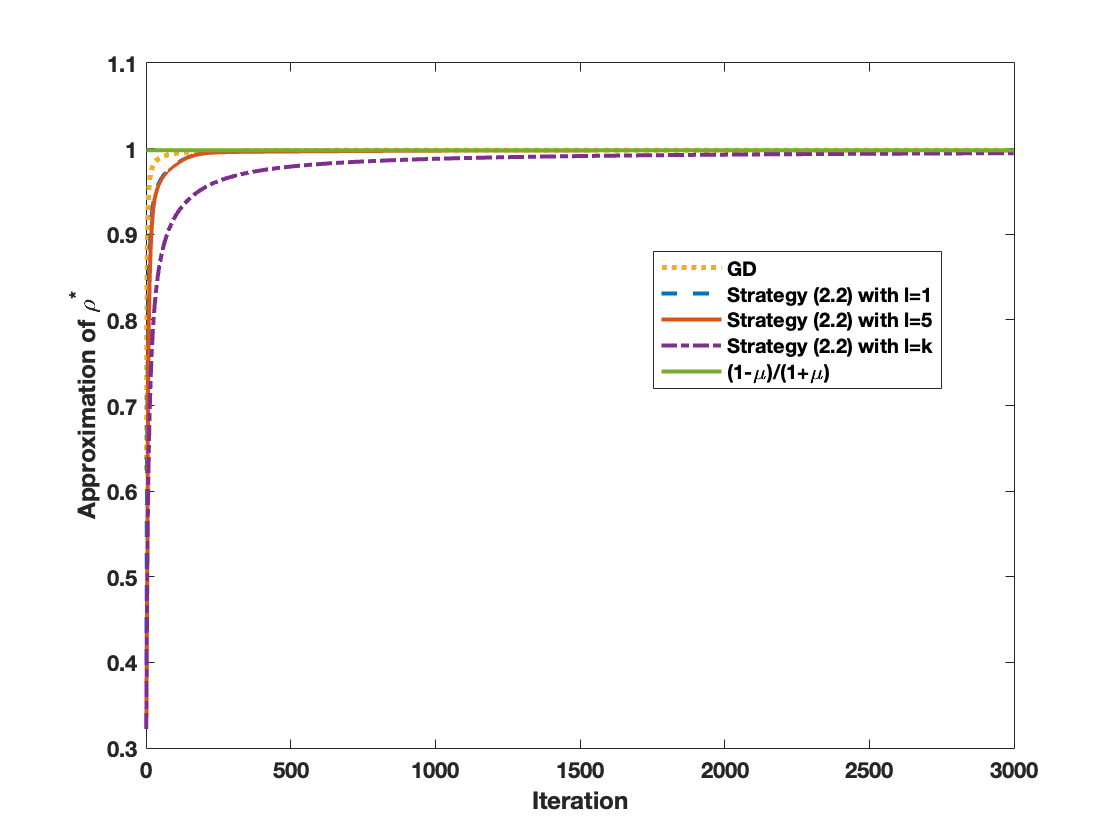}
    \end{minipage}
    \caption{Error (left) and estimated $\rho^*$ (right) for GD on a diagonal matrix with uniform eigenvalue distribution ($n=1000$).}
    \label{fig:uniform_gd}
\end{figure}

\begin{figure}[htbp]
    \centering
    \begin{minipage}[t]{0.49\textwidth}
        \centering
        \includegraphics[width=\textwidth]{ 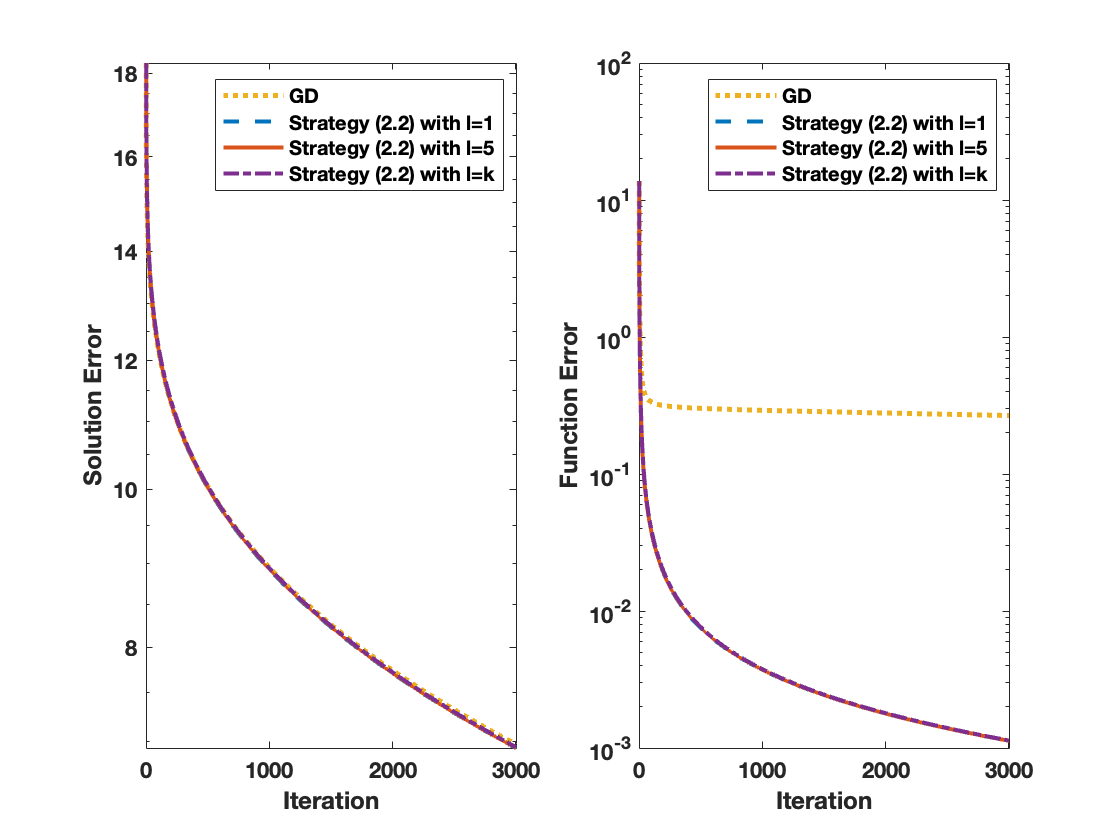}
    \end{minipage}%
    \hfill
    \begin{minipage}[t]{0.49\textwidth}
        \centering
        \includegraphics[width=\textwidth]{ 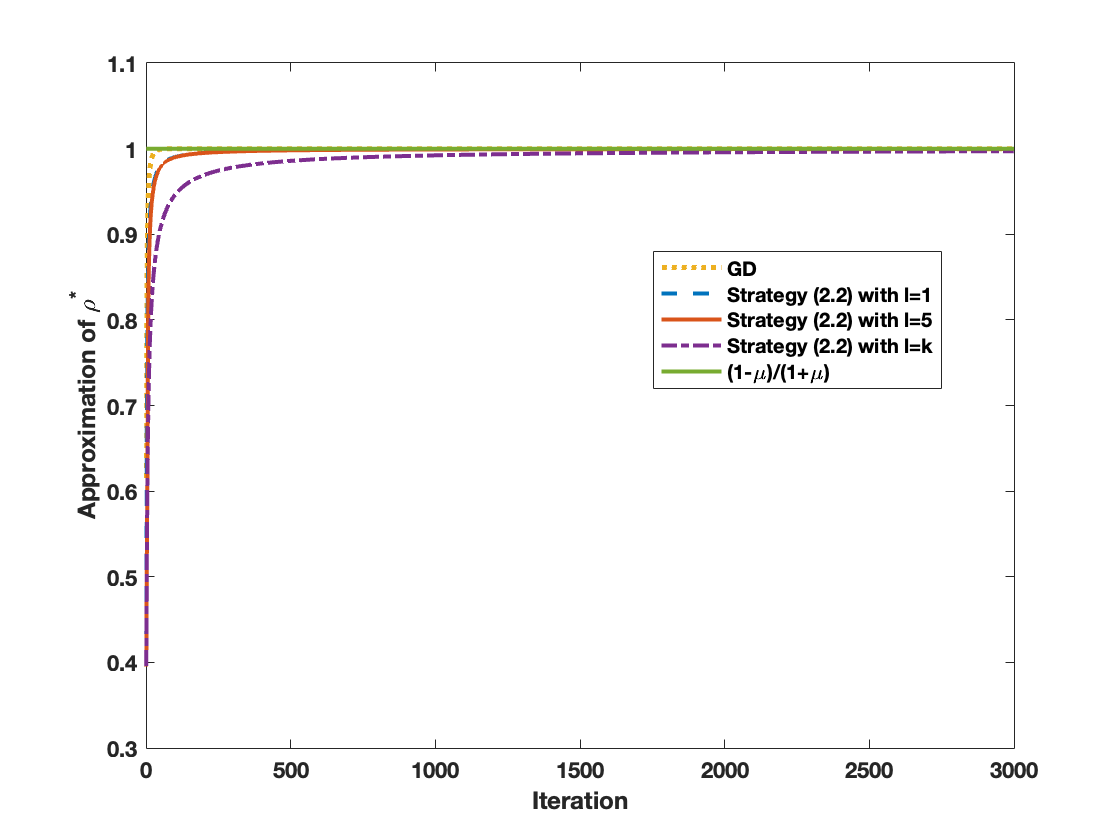}
    \end{minipage}
    \caption{Error (left) and estimated $\rho^*$ (right) for GD on a diagonal matrix with log-spaced eigenvalue distribution ($n=1000$).}
    \label{fig:log_gd}
\end{figure}

\begin{figure}[htbp]
    \centering
    \begin{minipage}[t]{0.49\textwidth}
        \centering
        \includegraphics[width=\textwidth]{ 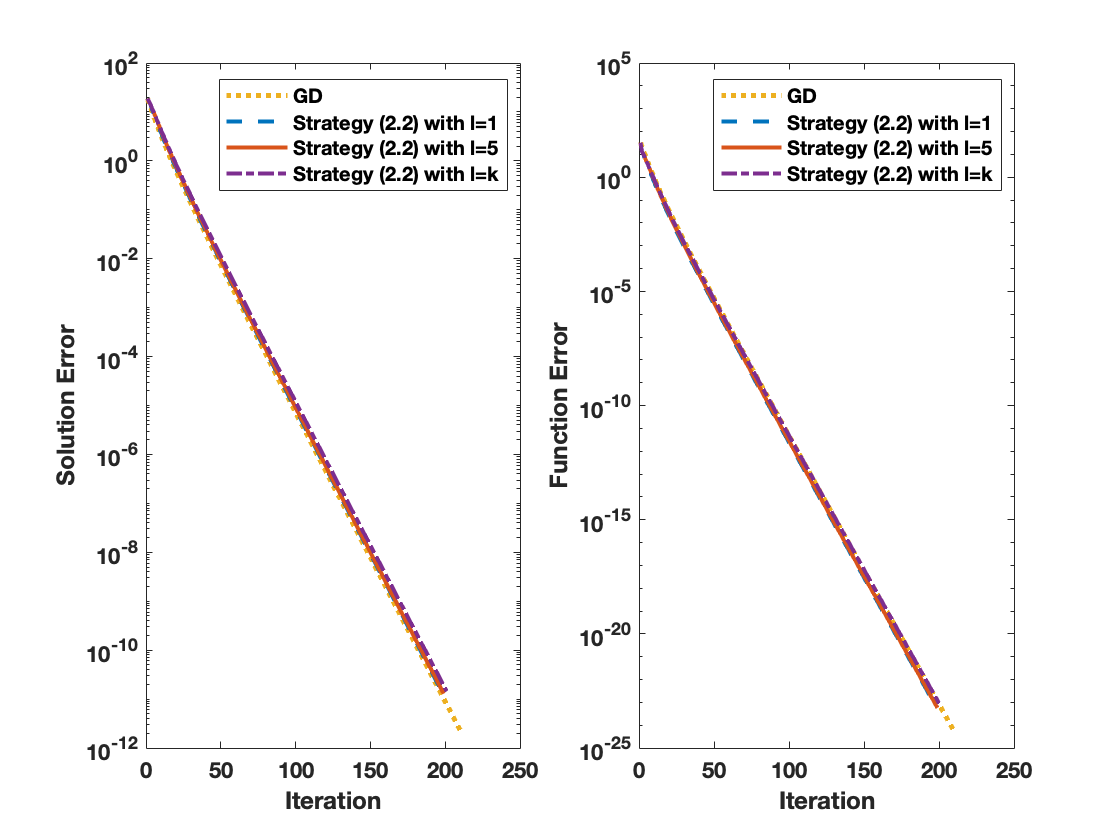}
    \end{minipage}%
    \hfill
    \begin{minipage}[t]{0.49\textwidth}
        \centering
        \includegraphics[width=\textwidth]{ 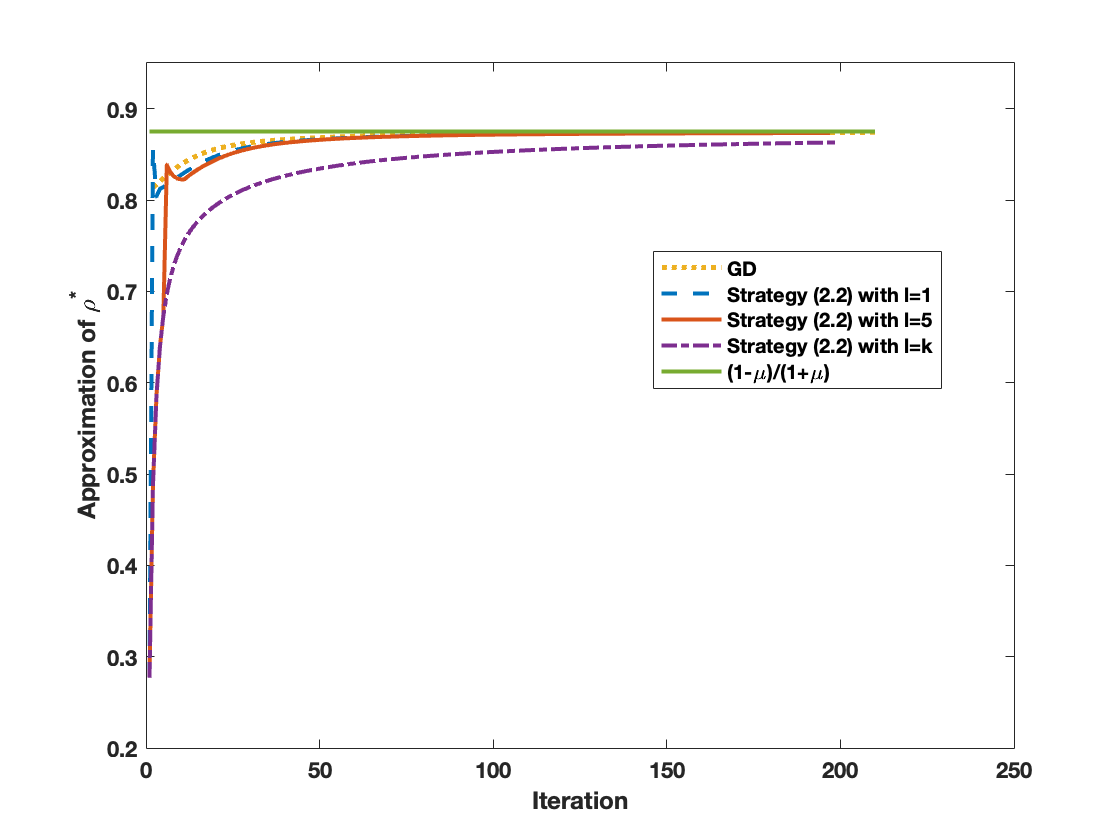}
    \end{minipage}
    \caption{Error (left) and estimated $\rho^*$ (right) for GD on a diagonal matrix with clustered eigenvalue distribution ($n=1000$).}
    \label{fig:cluster_gd}
\end{figure}

\Cref{fig:uniform_gd}-\ref{fig:cluster_gd} report the performance of GD and its adaptive variants (strategy \eqref{average_GD} with $l=1,5$, and $k$) on diagonal matrices with varying eigenvalue distributions. We see that for all tests, the estimated $\rho_k$ approaches $\rho_{\texttt{GD}}^*$ from below.  In both the uniform and log-spaced cases, the adaptive variants exhibit higher accuracy, reducing  the function value error by several orders of magnitude compared to standard GD. This is because the eigenvalues are relatively dispersed. Consequently, even when two iterates are close in the $\ell^2$ norm, their function values may differ significantly. In the clustered setting, the performance of the adaptive variants is comparable to that of standard GD. Furthermore, strategy \eqref{average_GD} with $l=k$, yields a smoother approximation of $\rho_k$, as it averages over all previous iterations. 

The comparison of NAG and its adaptive variants (strategy \eqref{average_NAG} with $l=1,5$, and $k$) are summarized in  \Cref{fig:uniform_na}-\ref{fig:cluster_nag}. We notice that strategy \eqref{average_NAG} with $l=1$  and 5 produce results that are comparable to standard NAG, demonstrating the effectiveness of the adaptive strategies. While strategy \eqref{average_NAG} with $l=k$ tends to perform slightly worse in the uniform and log-spaced cases. This is due to the fact that, in this case,  $\rho_k$ incorporating the entire history of iterations, leads to a more conservative choice. Moreover, the estimated convergence bound $\rho_k$ shows oscillations, but eventually converges to $\rho_{\texttt{NAG}}^*$. This behavior results from using only the most recent iteration, and hence the estimate is sensitive to short-term fluctuations 

\begin{figure}[htbp]
    \centering
    \begin{minipage}[t]{0.49\textwidth}
        \centering
        \includegraphics[width=\textwidth]{ 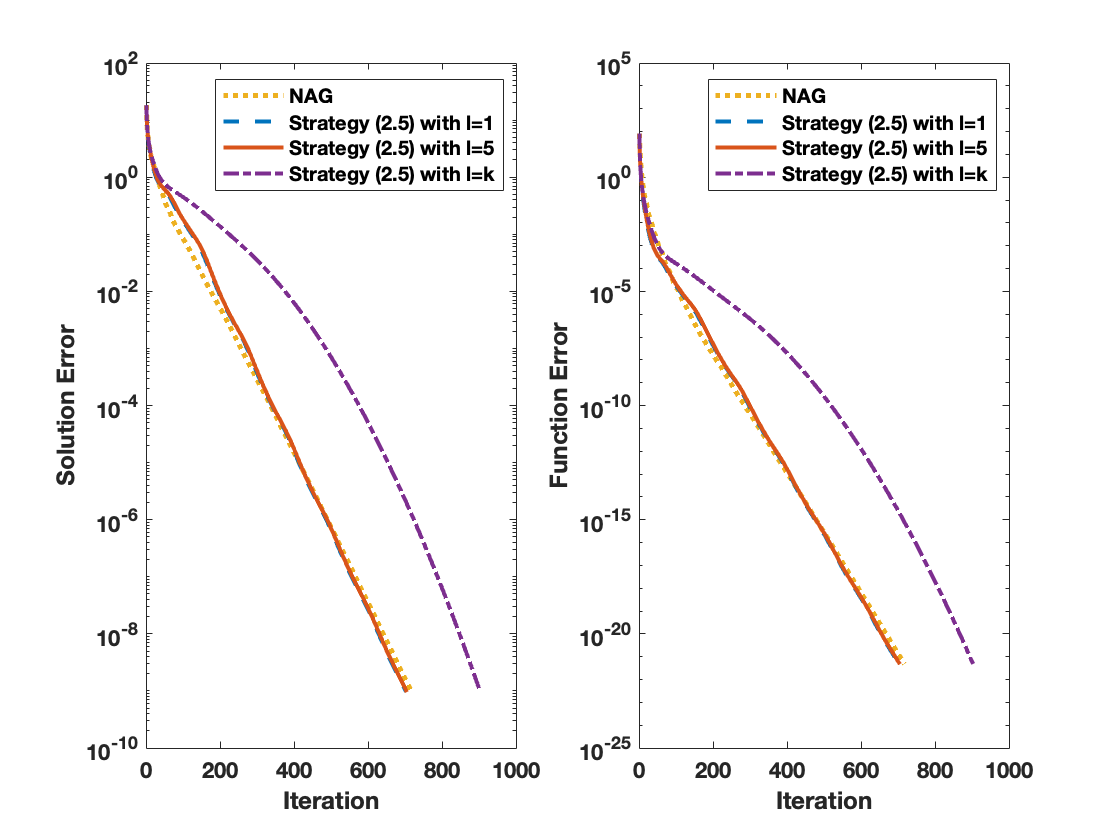}
    \end{minipage}%
    \hfill
    \begin{minipage}[t]{0.49\textwidth}
        \centering
        \includegraphics[width=\textwidth]{ 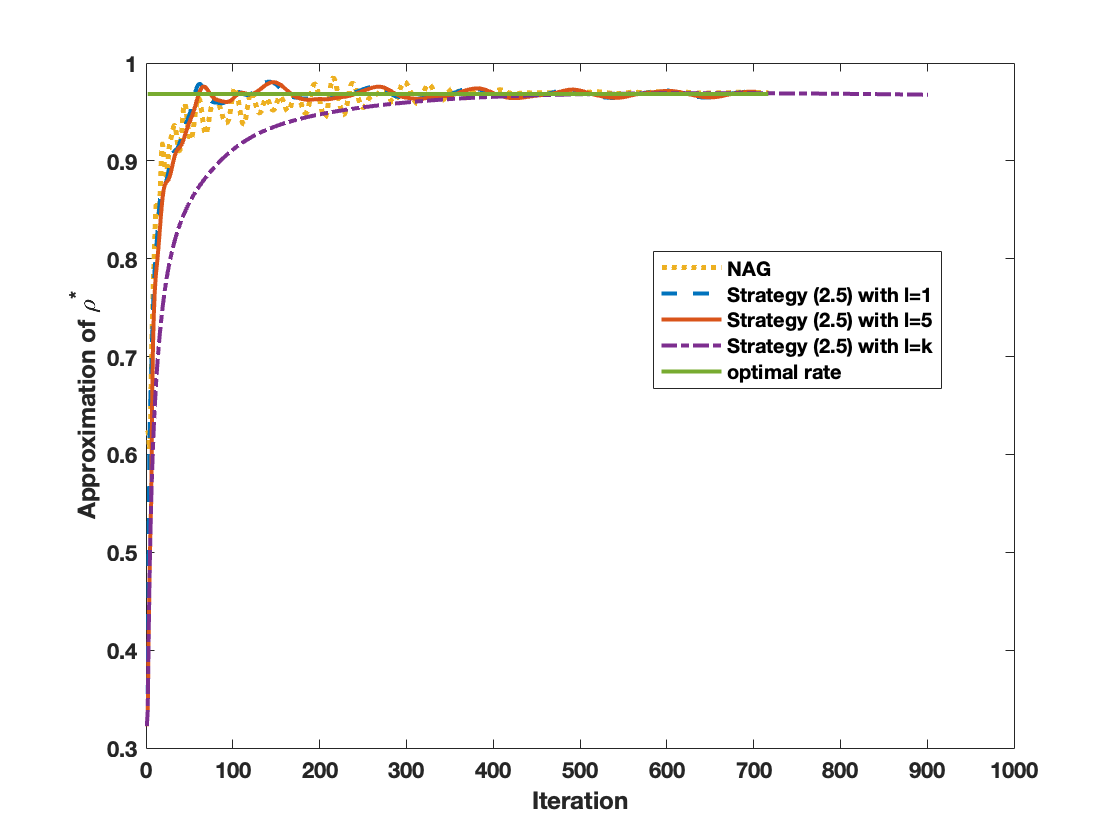}
    \end{minipage}
    \caption{Error (left) and estimated $\rho^*$ (right) for NAG on a diagonal matrix with uniform eigenvalue distribution ($n=1000$).}
    \label{fig:uniform_na}
\end{figure}
\begin{figure}[htbp]
    \centering
    \begin{minipage}[t]{0.49\textwidth}
        \centering
        \includegraphics[width=\textwidth]{ 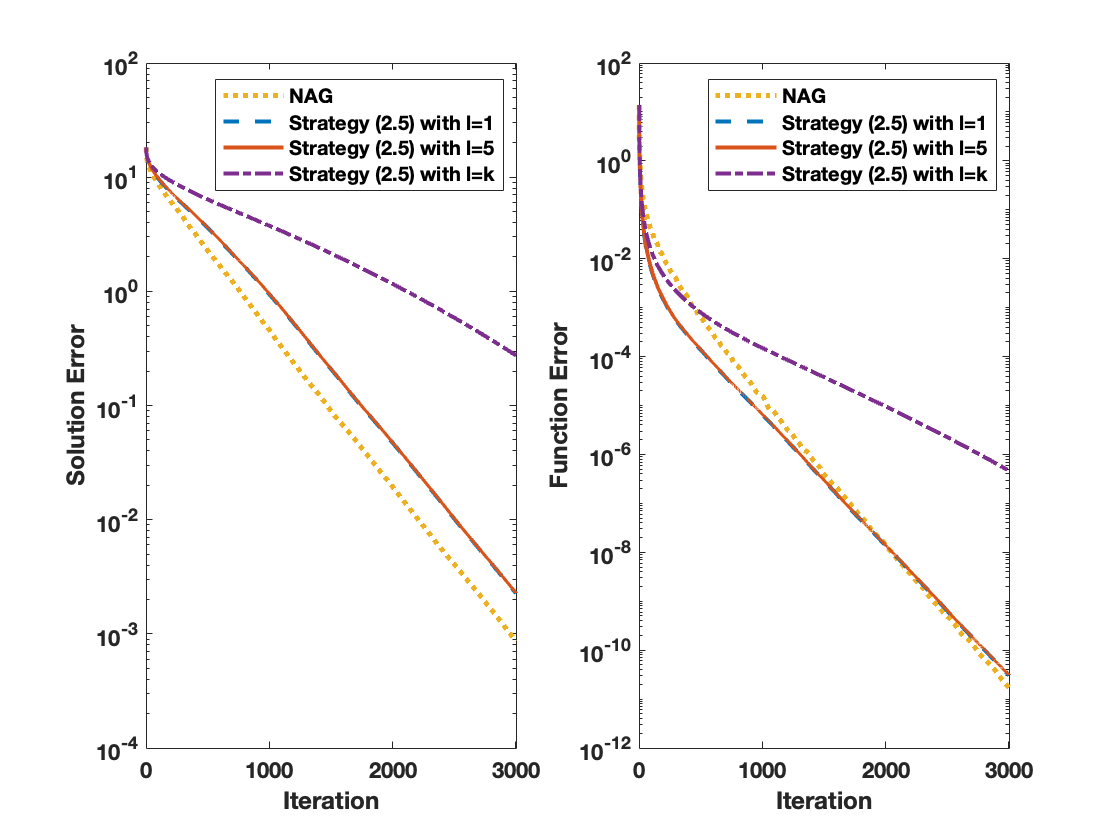}
    \end{minipage}%
    \hfill
    \begin{minipage}[t]{0.49\textwidth}
        \centering
        \includegraphics[width=\textwidth]{ 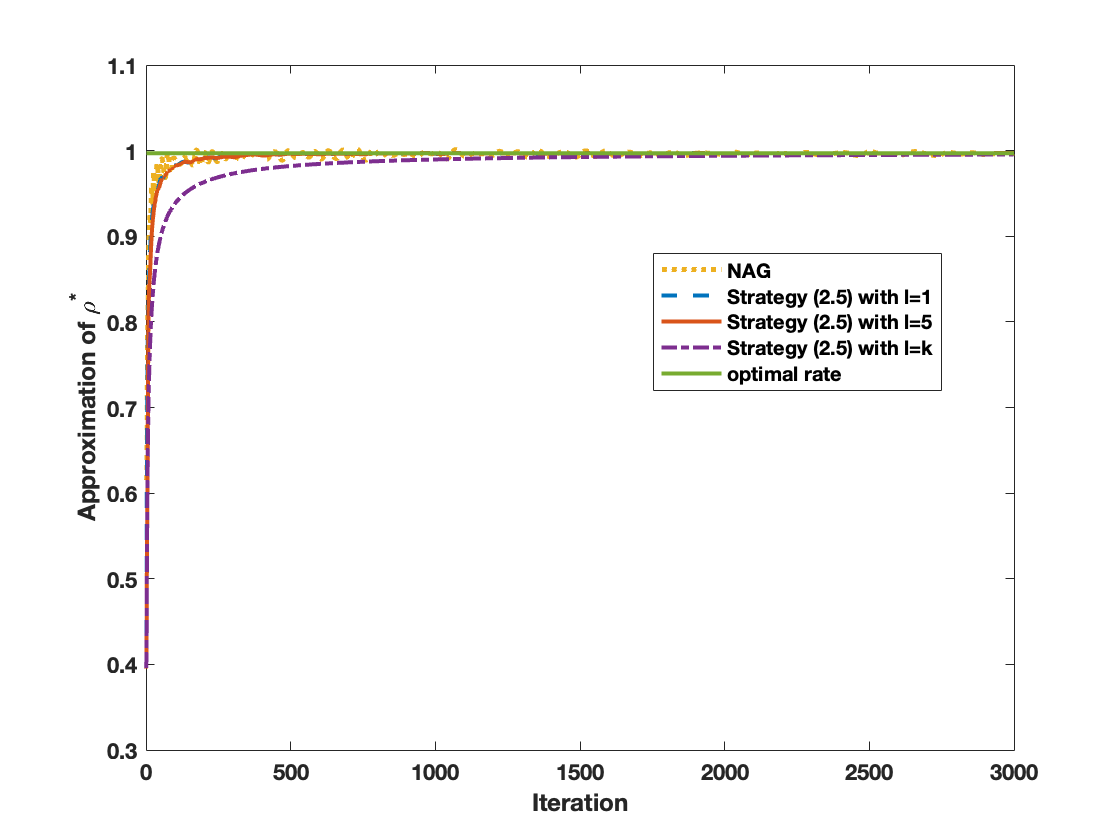}
    \end{minipage}
    \caption{Error (left) and estimated $\rho^*$ (right) for NAG on a diagonal matrix with log-spaced eigenvalue distribution ($n=1000$).}
    \label{fig:log_nag}
\end{figure}

\begin{figure}[htbp]
    \centering
    \begin{minipage}[t]{0.49\textwidth}
        \centering
        \includegraphics[width=\textwidth]{ 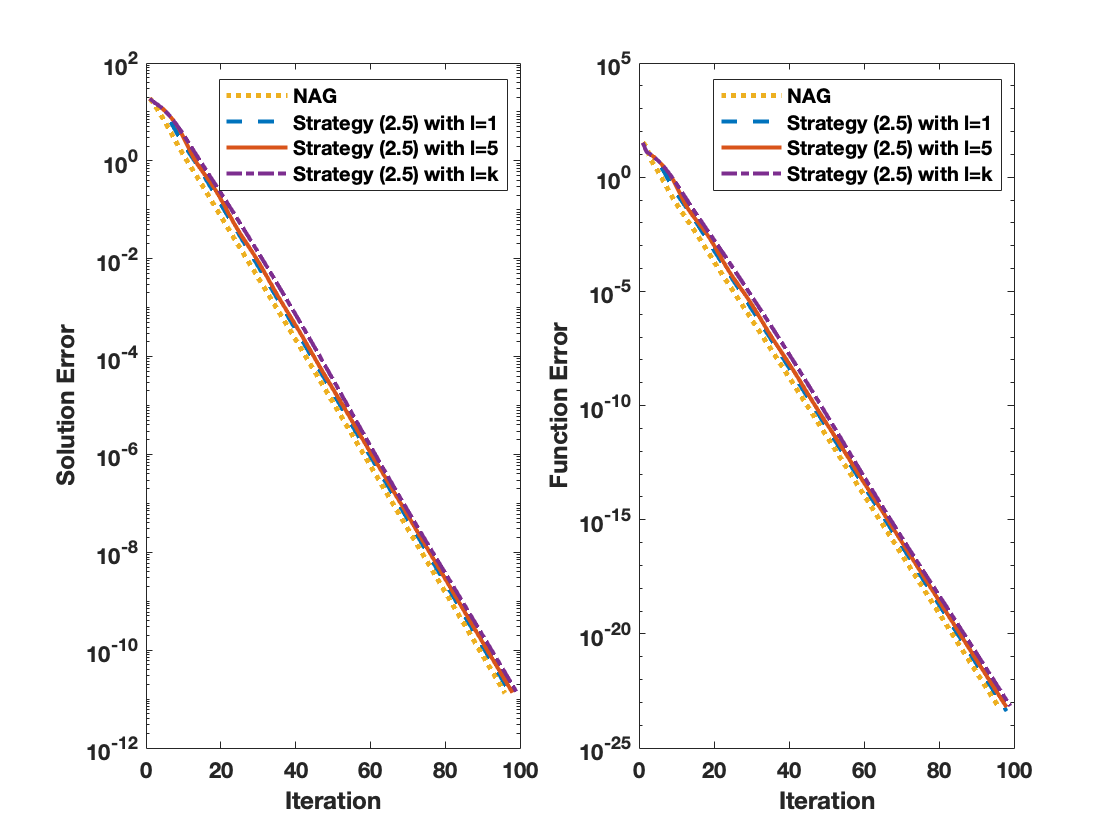}
    \end{minipage}%
    \hfill
    \begin{minipage}[t]{0.49\textwidth}
        \centering
        \includegraphics[width=\textwidth]{ 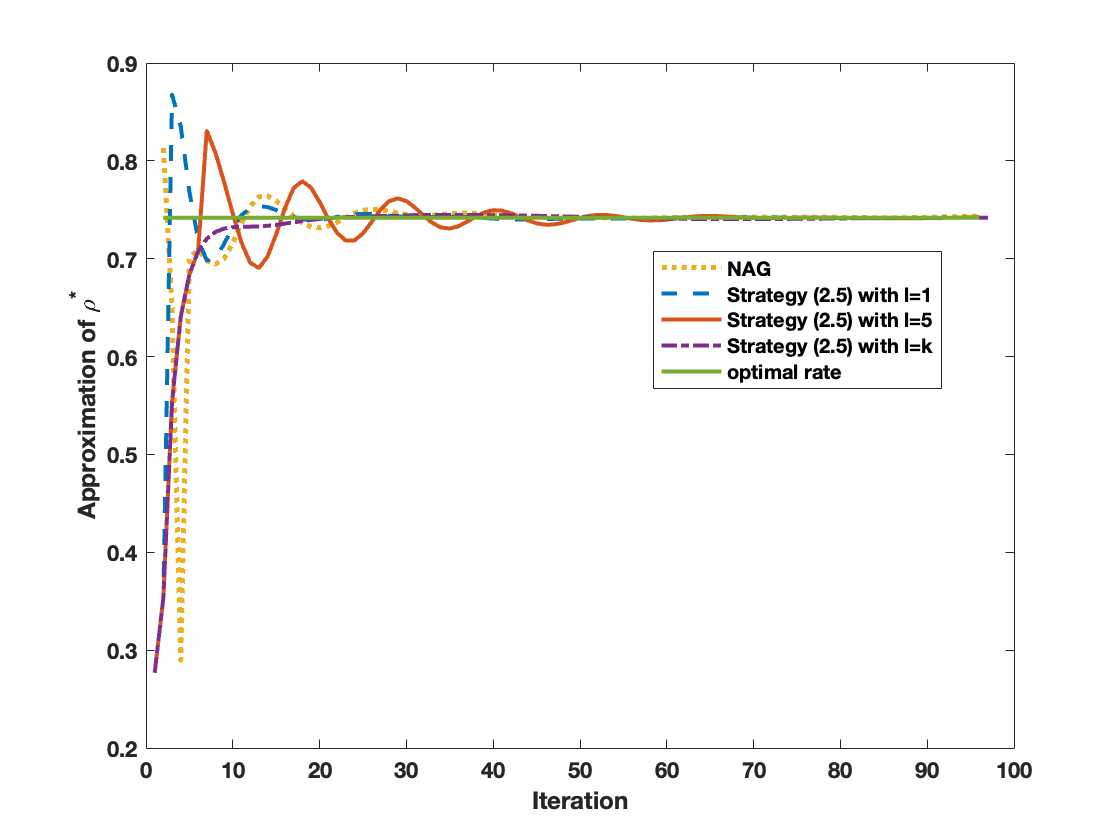}
    \end{minipage}
    \caption{Error (left) and estimated $\rho^*$ (right) for NAG on a diagonal matrix with clustered eigenvalue distribution ($n=1000$).}
    \label{fig:cluster_nag}
\end{figure}

The results of HB and the corresponding adaptive variants across various eigenvalue distributions are described in \Cref{fig:uniform_hb}-\ref{fig:cluster_hb}. Under the same number of iterations, both strategy \eqref{average_NAG} with $l=1$  and 5 achieve higher accuracy than the standard HB method for the uniform and log-spaced cases. Notably, for these two cases, the standard HB, as it is sensitive to parameter choices, shows some oscillation at the early stage. 
As shown in the clustered case, while strategy \eqref{average_NAG} with $l=1$  achieves higher accuracy, its estimation of $\rho_{\texttt{HB}}^*$ shows greater oscillations. Generally, strategy \eqref{average_NAG} with $l=5$ maintains both accuracy and stability, making it the most balanced choice among the HB variants.

\begin{figure}[htbp]
    \centering
    \begin{minipage}[t]{0.49\textwidth}
        \centering
        \includegraphics[width=\textwidth]{ 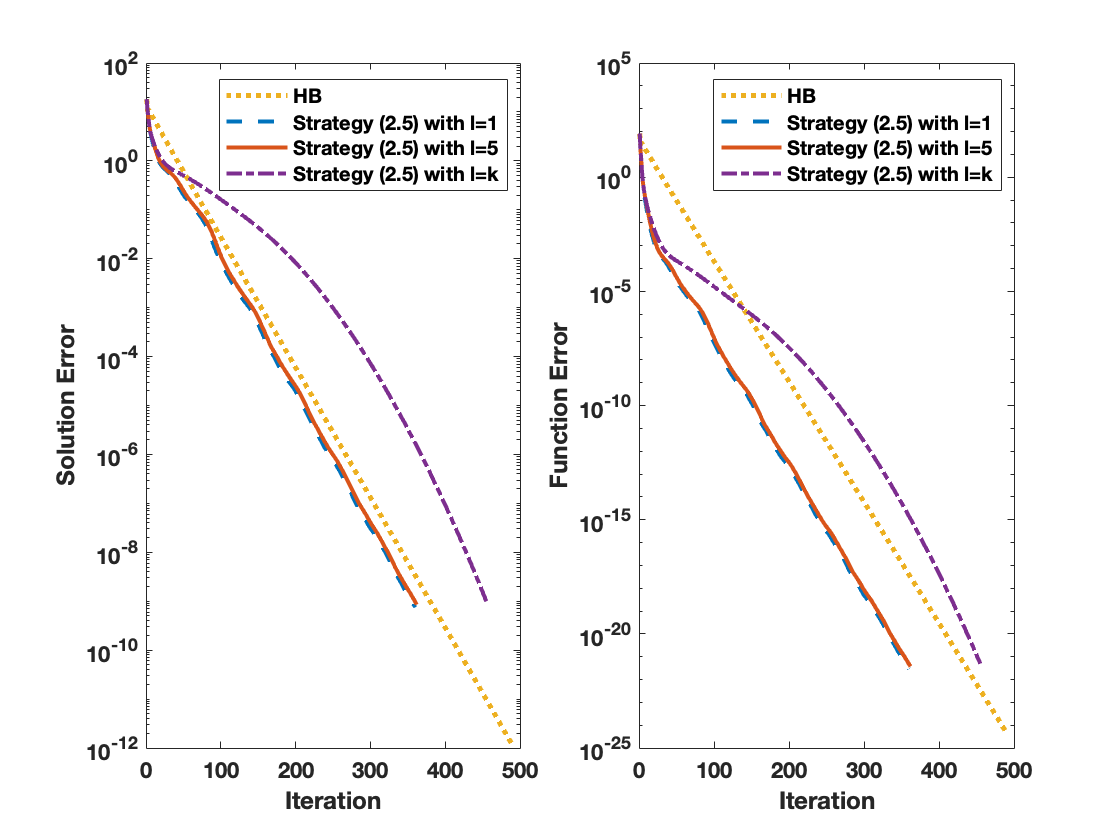}
    \end{minipage}%
    \hfill
    \begin{minipage}[t]{0.49\textwidth}
        \centering
        \includegraphics[width=\textwidth]{ 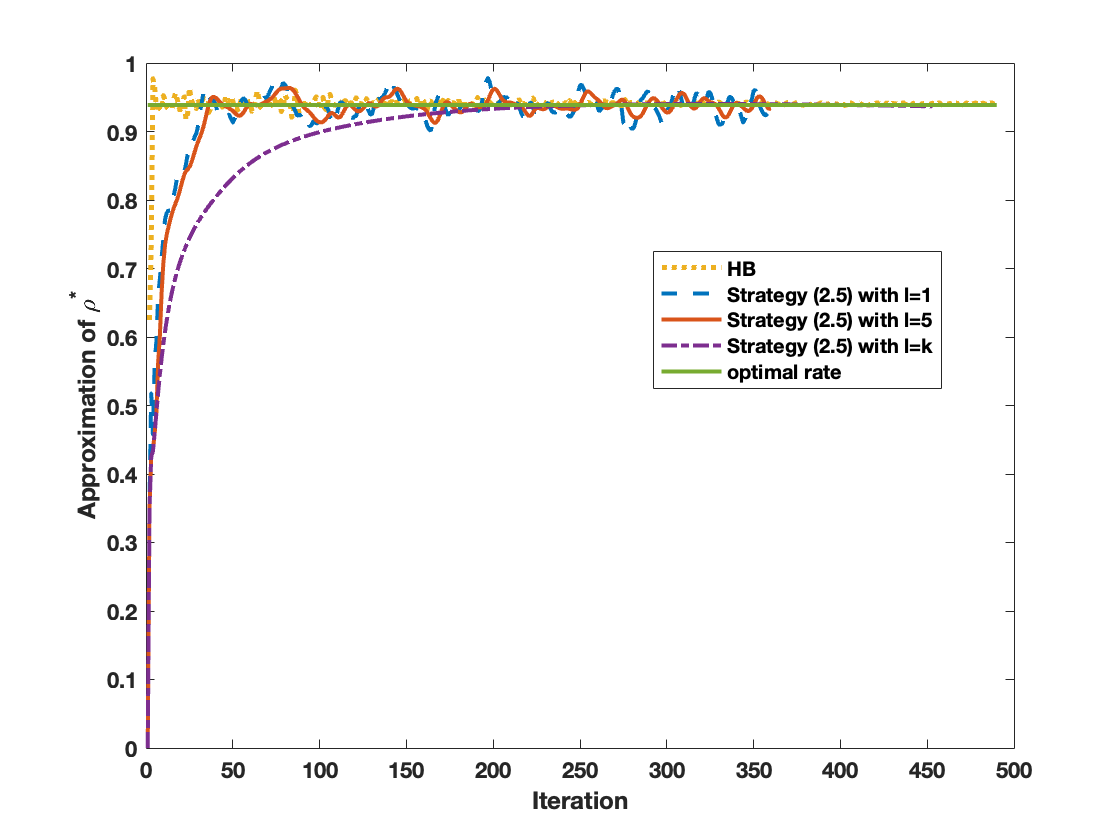}
    \end{minipage}
    \caption{Error (left) and estimated $\rho^*$ (right) for HB on a diagonal matrix with uniform eigenvalue distribution ($n=1000$).}
    \label{fig:uniform_hb}
\end{figure}
\begin{figure}[htbp]
    \centering
    \begin{minipage}[t]{0.49\textwidth}
        \centering
        \includegraphics[width=\textwidth]{ 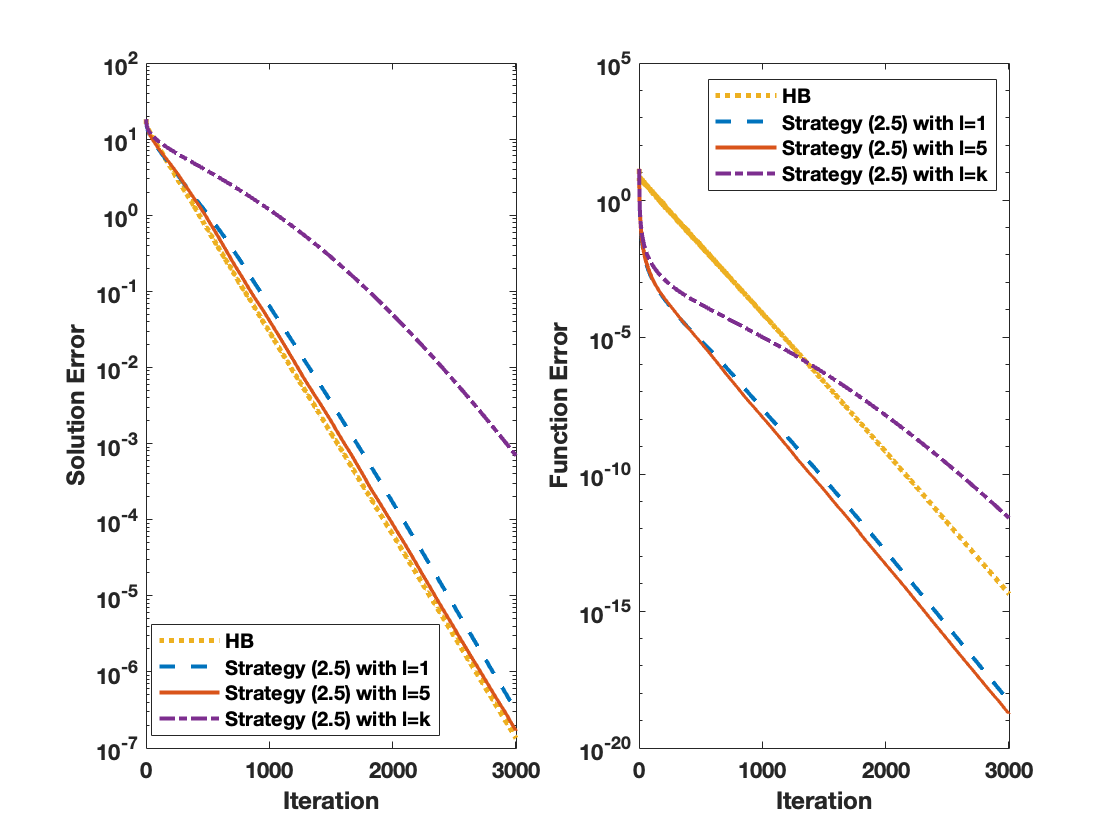}
    \end{minipage}%
    \hfill
    \begin{minipage}[t]{0.49\textwidth}
        \centering
        \includegraphics[width=\textwidth]{ 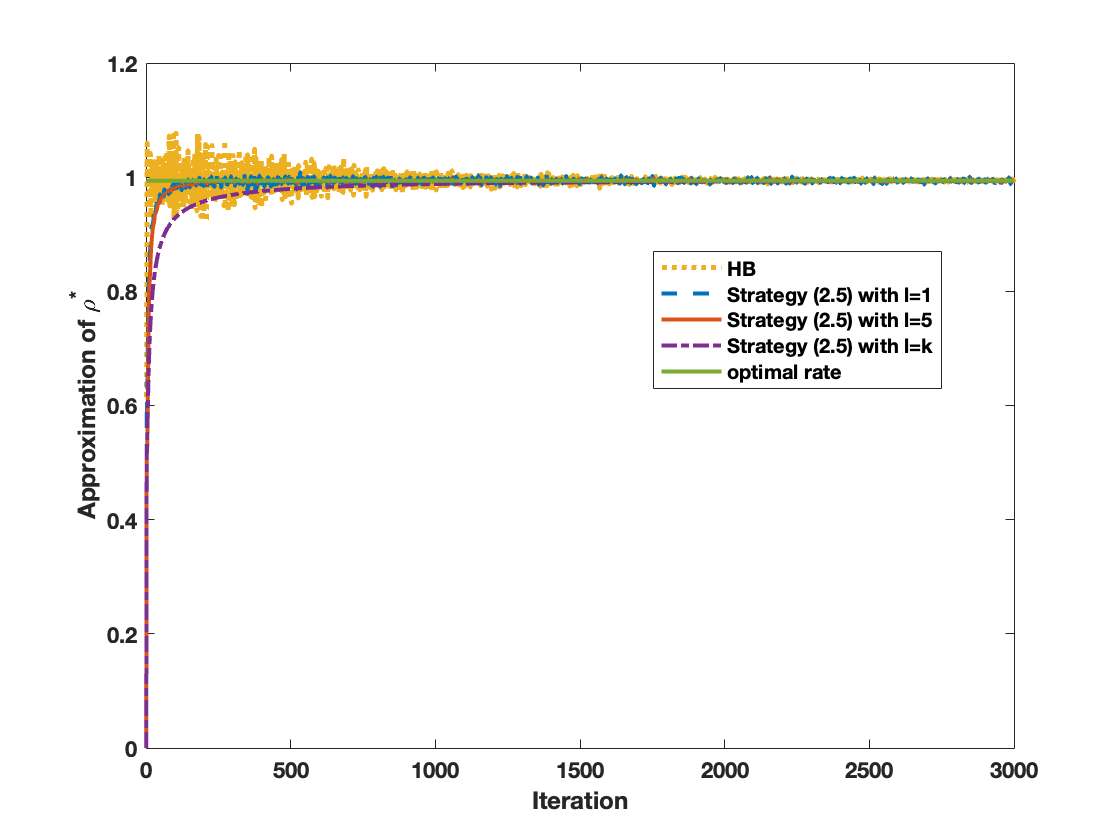}
    \end{minipage}
    \caption{Error (left) and estimated $\rho^*$ (right) for HB on a diagonal matrix with log-spaced eigenvalue distribution ($n=1000$).}
    \label{fig:log_hb}
\end{figure}

\begin{figure}[htbp]
    \centering
    \begin{minipage}[t]{0.49\textwidth}
        \centering
        \includegraphics[width=\textwidth]{ 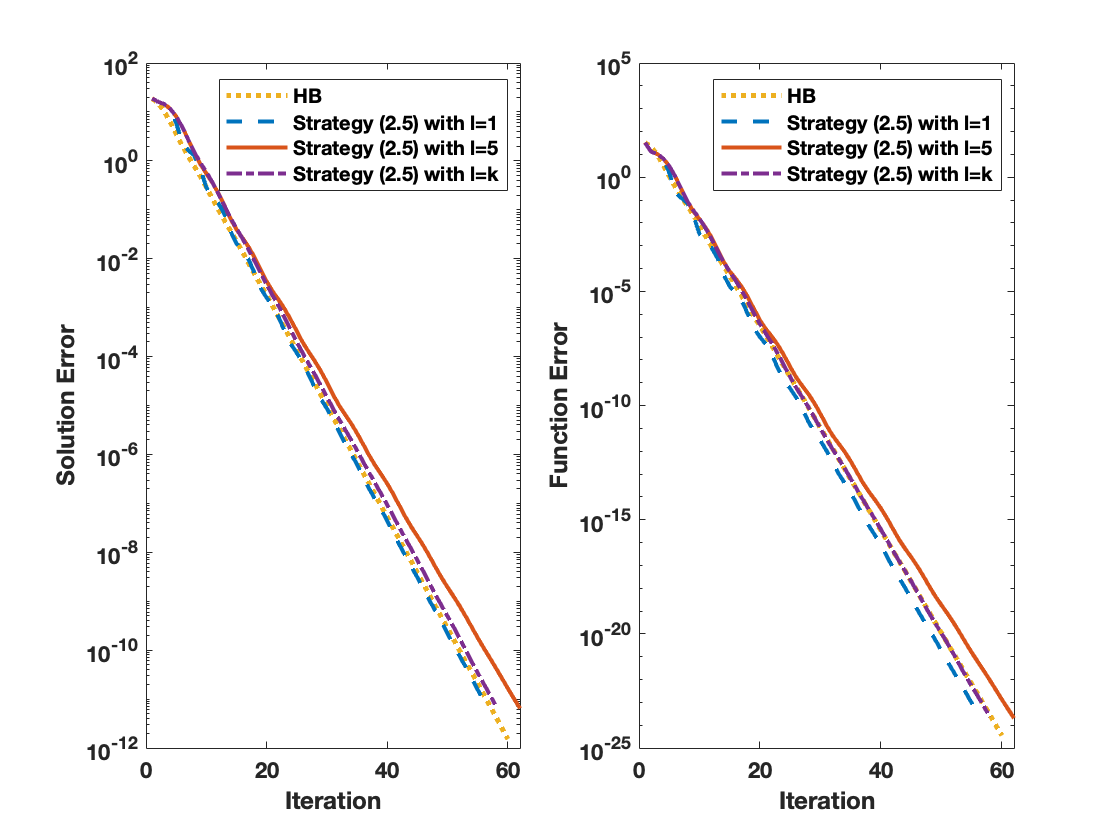}
    \end{minipage}%
    \hfill
    \begin{minipage}[t]{0.49\textwidth}
        \centering
        \includegraphics[width=\textwidth]{ 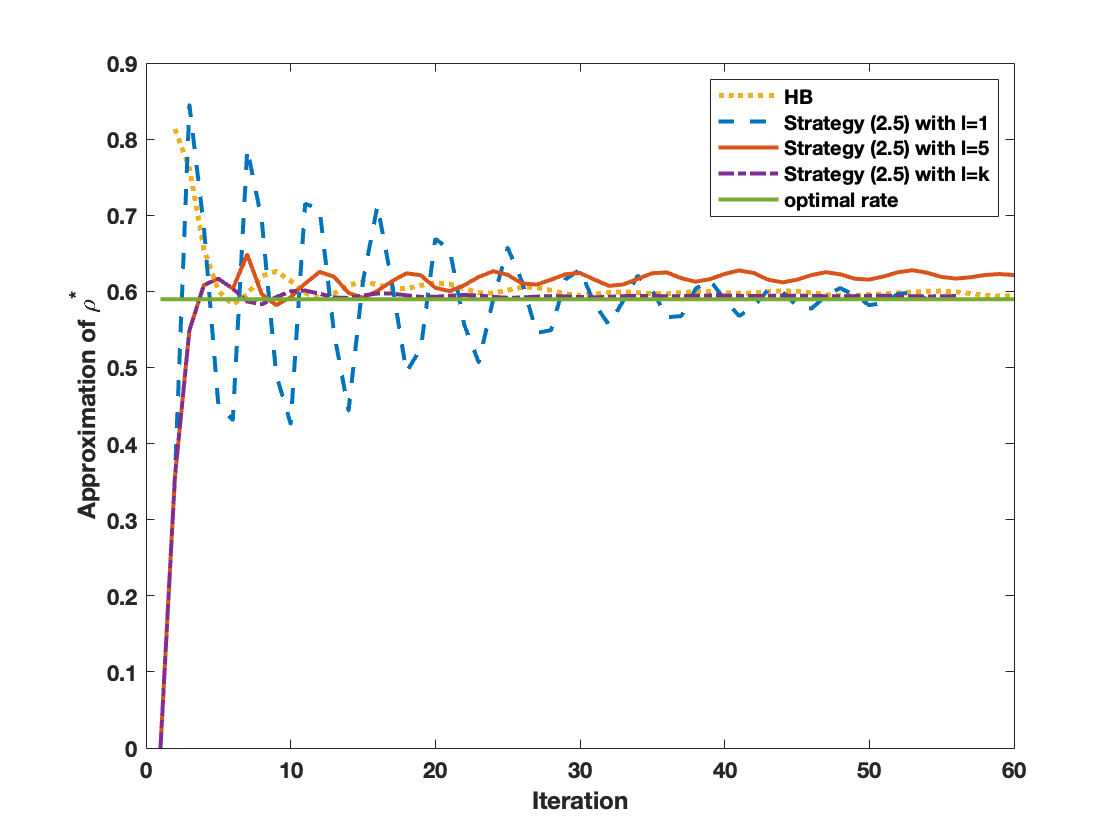}
    \end{minipage}
    \caption{Error (left) and estimated $\rho^*$ (right) for HB on a diagonal matrix with clustered eigenvalue distribution ($n=1000$).}
    \label{fig:cluster_hb}
\end{figure}





\subsection{Logistic regression problem with regularization} 
Logistic regression is a fundamental statistical method used for binary classification, with wide range applications in fields such as finance \cite{broby2022use,hasan2021does}, medical \cite{zabor2022logistic,schober2021logistic}, and marketing \cite{constantin2015using,burinskiene2007application}.  Consider a dataset consists of $p$ samples, each characterized by  $n$ features. The feature matrix  $A \in \mathbb{R}^{n \times p}$ is constructed such that each column represents a distinct sample. For each sample $i$, the response is binary, taking values $b_i\in \{-1,1\}$. We define the vector of all labels as $\bm{b} = [b_1, b_2, \dots, b_p]^T$. Logistic regression models the conditional probability of the label $b_i$ given the input $A_{:,i}$ (the $i$-th column of $A$) via:
$$P(b_i \, | \, A_{:,i}) = \mathrm{sigm}(b_i (\bm{x}^T A_{:,i})).$$
Here the sigmoid function is $\mathrm{sigm}(z):= \frac{1}{1+e^{-z}}$. 
To estimate the weight vector $\bm{x} \in \mathbb{R}^n$, we consider minimizing the negative log-likelihood: 
$$f(\bm{x}) = \frac{\xi}{2} \bm{x}^T \bm{x} + \sum_{i=1}^{p} \log(1 + e^{-b_i (\bm{x}^TA_{:,i})}),$$
where $\xi$ is the regularization parameter, ensuring that the optimization problem is strongly convex and preventing overfitting. The gradient and Hessian for $f(\bm{x})$ are given by:
\begin{align*}
\nabla f(\bm{x}) = \ & \xi \bm{x} + \sum_{i=1}^{p} \frac{e^{-b_i (\bm{x}^TA_{:,i})}}{1 + e^{-b_i (\bm{x}^TA_{:,i})}} (-b_i A_{:,i})\\
\nabla^2 f(\bm{x})
= \ & \xi I + \sum_{i=1}^{p} \mathrm{sigm}(b_i (\bm{x}^TA_{:,i}))(1 - \mathrm{sigm}(b_i (\bm{x}^TA_{:,i})) A_{:,i} A_{:,i}^T.
\end{align*}

Note that $0 < \mathrm{sigm}(z)(1-\mathrm{sigm}(z)) \leq \frac{1}{4}$ 
, so the eigenvalues of the Hessian matrix lie in the range $[\xi, \xi + \frac{1}{4} \lambda_{\max}(AA^T)]$. 
This motivates us to approximate the minimum eigenvalue and the maximum eigenvalue of the Hessian by $\tilde{\mu}=\xi$ and $\tilde{L} =  \xi + \frac{1}{4} \|AA^T\|_1$, respectively. 
Thus, the original algorithms apply the fixed $\tilde{\mu}$ and $\tilde{L}$. For the adaptive method, with $\tilde{L}$, the update rules are given by:
    \begin{align*}
    \mu_k =
    \begin{cases}
    \displaystyle \frac{1 - \rho_k}{1 + \rho_k} \tilde{L}, & \text{for GD} \\
    (1 - \rho_k)^2 \tilde{L}, & \text{for NAG} \\
    \left( \frac{1 - \rho_k}{1 + \rho_k} \right)^2 \tilde{L}, & \text{for HB}. 
    \end{cases}
    \end{align*}
In addition, for both NAG and HB, the first iteration is a GD step with a step size of $\alpha = \frac{2}{\tilde{\mu}+\tilde{L}}$.

In this test, we focus on the classical setting where $p>n$, often arising when data are abundant. Here, 
$A \in \mathbb{R}^{n \times p} $ is generated with entries sampled from the standard normal distribution, and the binary label vector $\bm{b} \in \{-1, 1\}^p$ is drawn from a Bernoulli distribution with $0 \mapsto -1$. We use $n = 500$, $p = 2000$, $\xi=0.1$ and $\bm{x}_0 = \bm{0}$. 
The stopping criterion for all tested algorithms is either a maximum of \texttt{maxIt} or iterations or a relative gradient norm $\|\nabla f(\bm{x}_k)\| / \|\nabla f(\bm{x}_0)\| < \texttt{tol}$ with \texttt{maxIt} = 350 and \texttt{tol} = $10^{-6}$. 
To get a reference solution $\bm{x}_{\text{ref}}$, we employ the NAG for convex objectives \cite{su2016differential} with a higher precision ($\texttt{maxIt}=10^5$, $\texttt{tol} = 10^{-12}$). 
The performance is then evaluated using the solution error $\|\bm{x}_k - \bm{x}_{\text{ref}}\|$, and the function error $f(\bm{x}_k) - f(\bm{x}_{\text{ref}})$.  


In \Cref{fig:logistic_gd}–\ref{fig:logistic_hb}, we show the results of GD, NAG, and HB and their adaptive variants. Since the optimal rate bound $\rho^*$ changes at each iteration, we estimate it dynamically for standard methods using the strategy \eqref{average_GD} and \eqref{average_NAG} with $l=1$, respectively. We observe that for GD, the adaptive strategies achieve similar results to  that of the original method with fixed $\tilde{L}$ and $\tilde{\mu}$ in both accuracy and iteration count. For NAG, strategy \eqref{average_NAG} with $l=k$ performs the worst. In contrast, strategy \eqref{average_NAG} with $l=1$ and 5  achieve accuracy of the same order of magnitude as NAG with fixed $\tilde{L}$ and $\tilde{\mu}$ but using fewer iterations.  This is because adaptive algorithms dynamically adapt to local curvature information of the objective.   On the other hand, here NAG only relies on fixed parameters derived from $\tilde{\mu}$ and $\tilde{L}$, which can be poor  estimations. 
A similar phenomenon is observed for HB, where adaptive variants show better results than the standard one.

\begin{figure}[htbp]
    \centering
    \begin{minipage}[t]{0.49\textwidth}
        \centering
        \includegraphics[width=\textwidth]{ 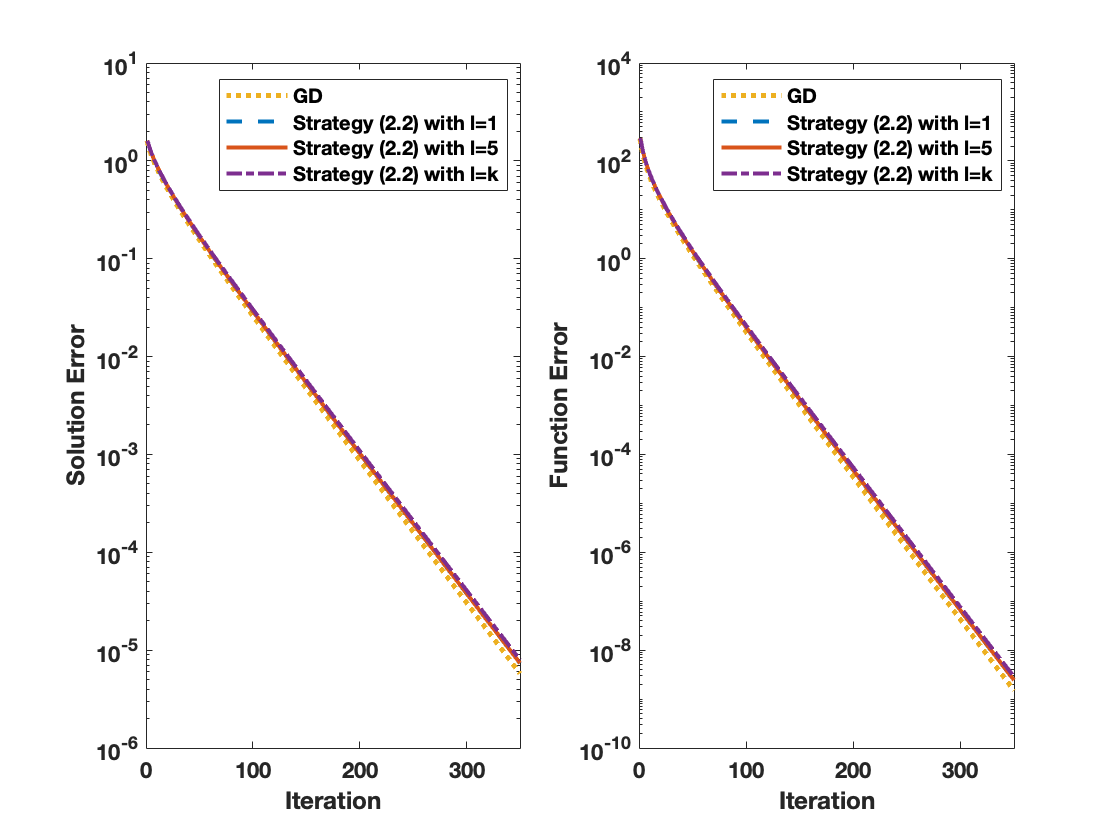}
    \end{minipage}%
    \hfill
    \begin{minipage}[t]{0.49\textwidth}
        \centering
        \includegraphics[width=\textwidth]{ 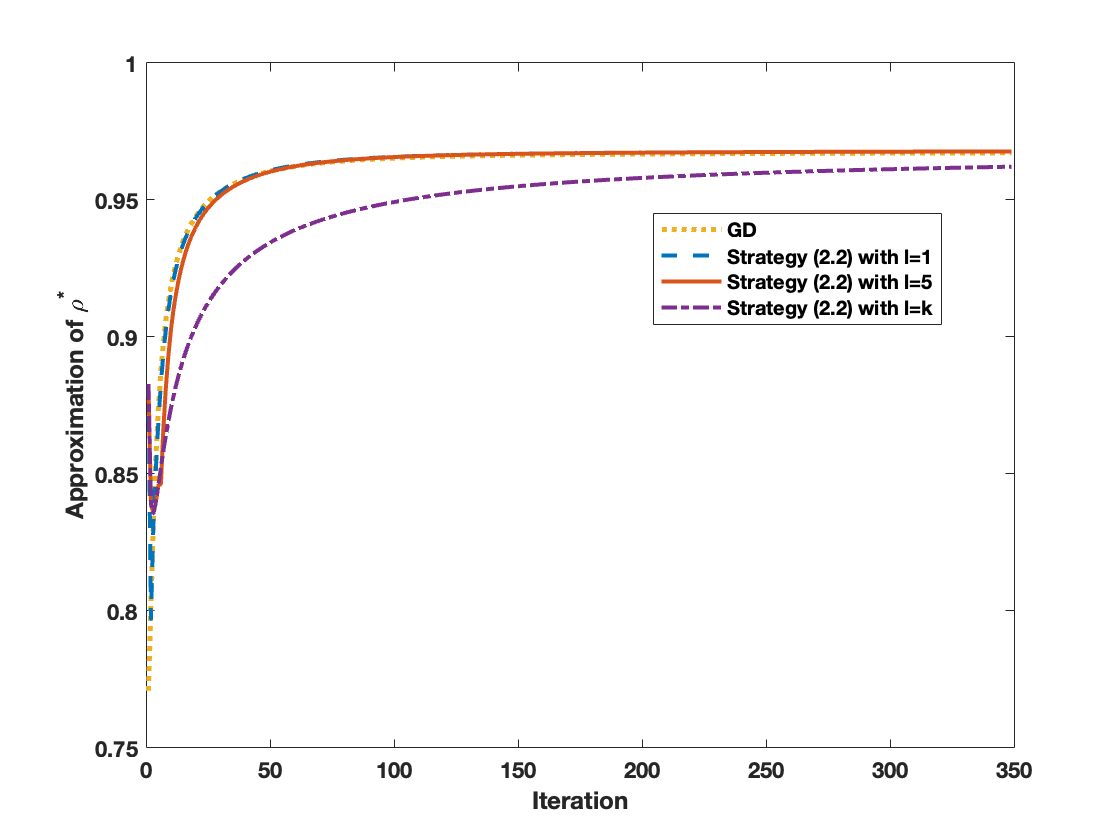}
    \end{minipage}
    \caption{Error (left) and estimated $\rho^*$ (right) for GD and its variants on the logistic regression problem.}
    \label{fig:logistic_gd}
\end{figure}

\begin{figure}[htbp]
    \centering
    \begin{minipage}[t]{0.49\textwidth}
        \centering
        \includegraphics[width=\textwidth]{ 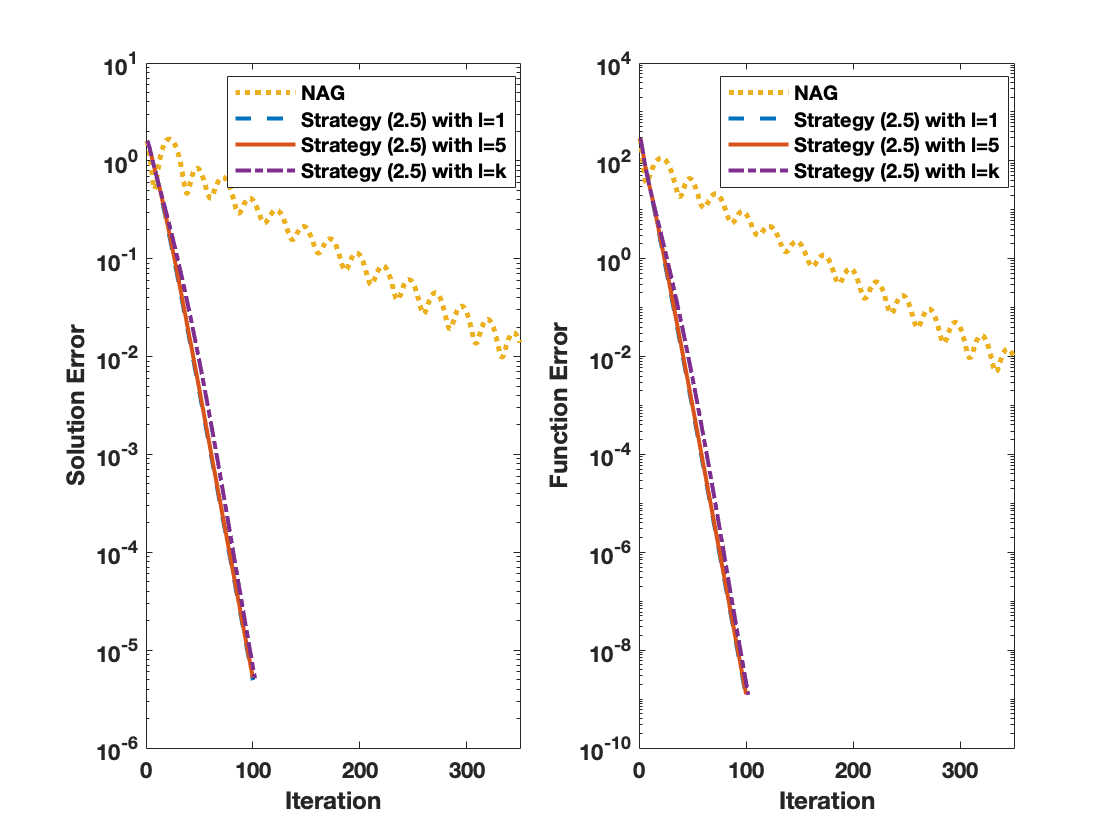}
    \end{minipage}%
    \hfill
    \begin{minipage}[t]{0.49\textwidth}
        \centering
        \includegraphics[width=\textwidth]{ 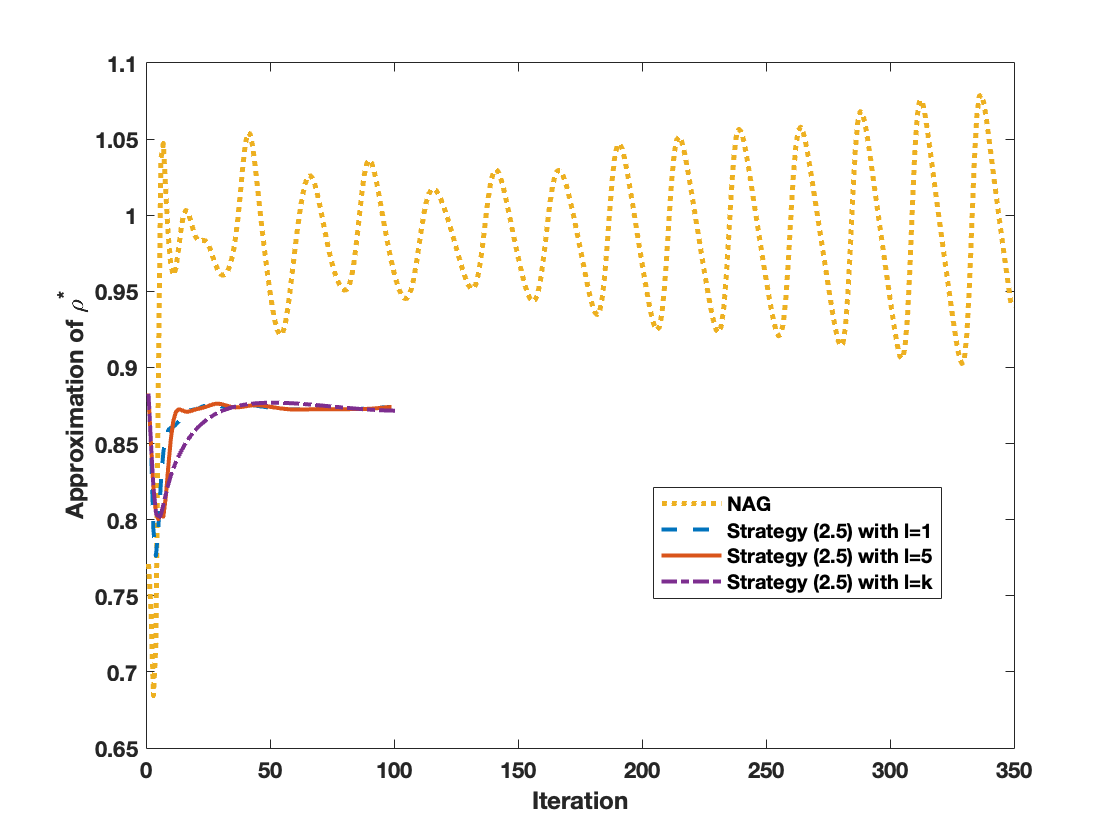}
    \end{minipage}
    \caption{Error (left) and estimated $\rho^*$ (right) for NAG and its variants on the logistic regression problem.}
    \label{fig:logistic_nag}
\end{figure}
 
\begin{figure}[htbp]
    \centering
    \begin{minipage}[t]{0.49\textwidth}
        \centering
        \includegraphics[width=\textwidth]{ 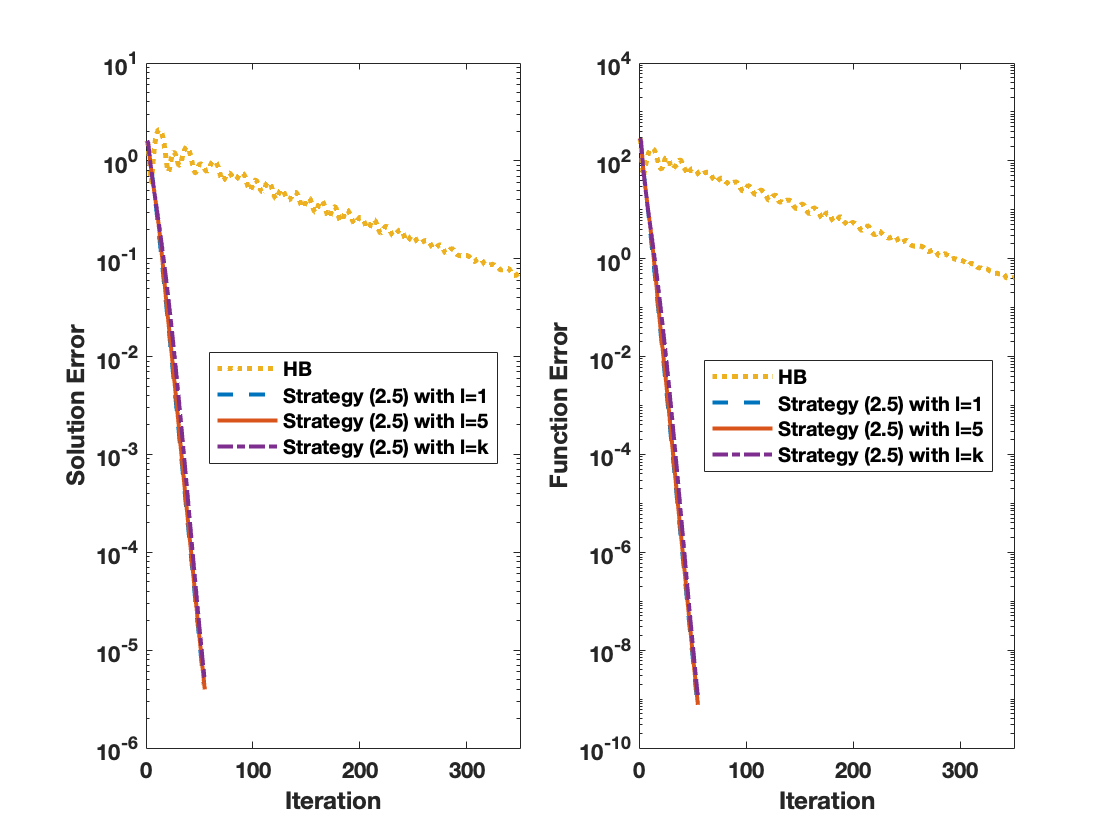}
    \end{minipage}%
    \hfill
    \begin{minipage}[t]{0.49\textwidth}
        \centering
        \includegraphics[width=\textwidth]{ 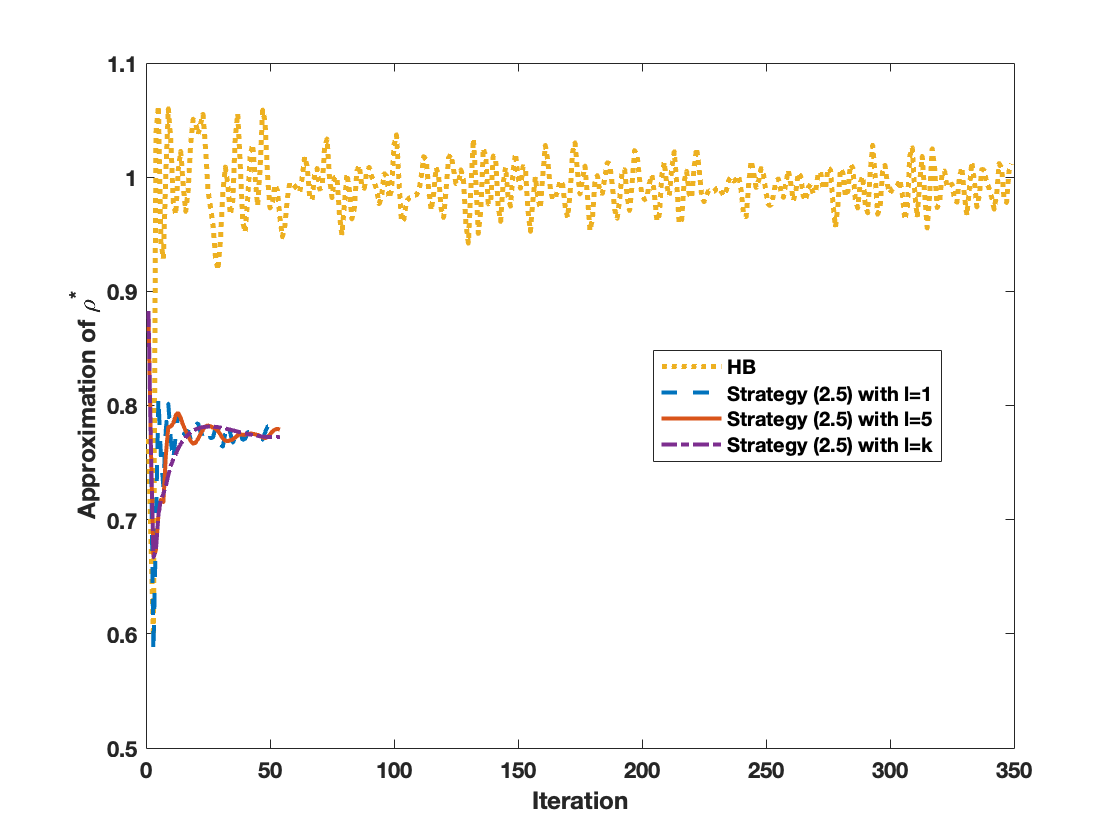}
    \end{minipage}
    \caption{Error (left) and estimated $\rho^*$ (right) for HB and its variants on the logistic regression problem.}
    \label{fig:logistic_hb}
\end{figure}

To further examine whether the observed behavior arises from inaccurate parameter choices or from the effectiveness of the adaptive algorithms, we use power iteration to estimate $L$, and in the original methods we use $L$ with shifted power iteration (with tolerance $10^{-6}$ and max 500 iterations) to estimate $\mu$ at each iteration. Then we arrive at the following results, as shown in \Cref{fig:logistic_gd_new}-\ref{fig:logistic_hb_new}. We observe that the adaptive algorithms and the original methods with estimated $L$ and $\mu$ 
exhibit similar performance. Compared to the previous tests with fixed $\tilde{L}$ and $\tilde{\mu}$, fewer iterations are required.  Nonetheless, as mentioned, estimating $L$ at each step increases the computational cost. 

\begin{figure}[htbp]
    \centering
    \begin{minipage}[t]{0.49\textwidth}
        \centering
        \includegraphics[width=\textwidth]{ 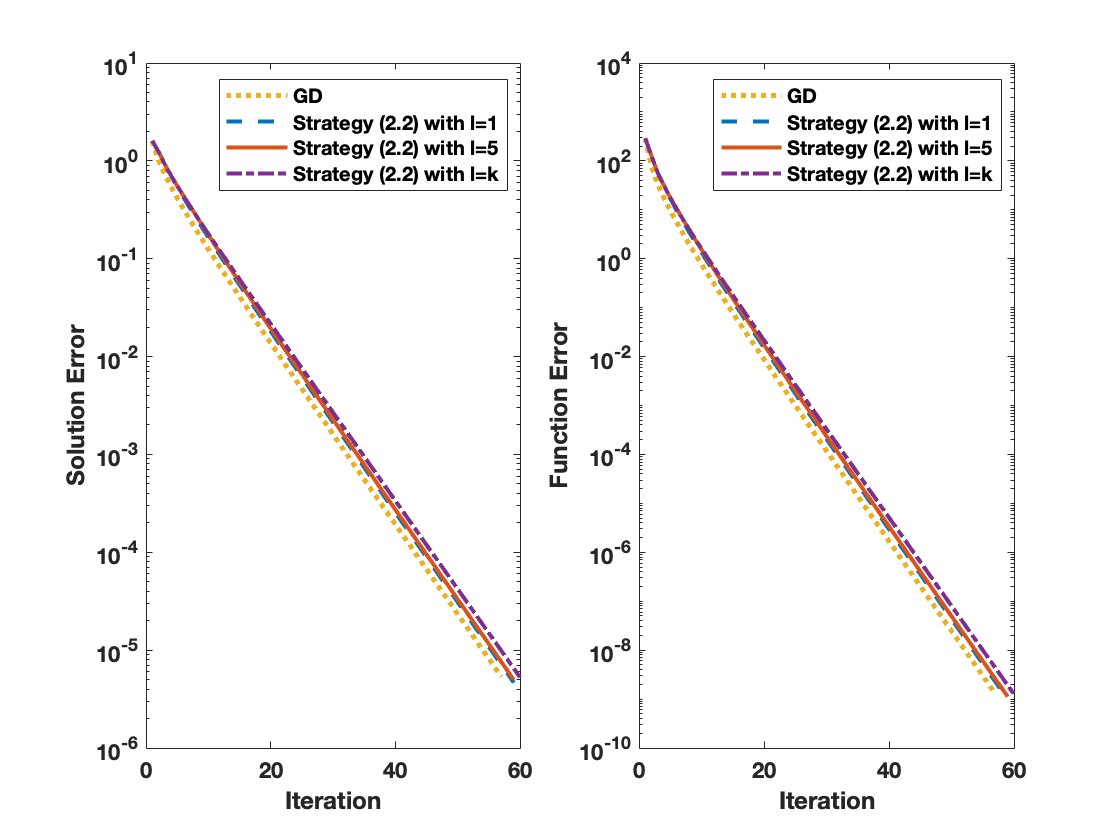}
    \end{minipage}%
    \hfill
    \begin{minipage}[t]{0.49\textwidth}
        \centering
        \includegraphics[width=\textwidth]{ 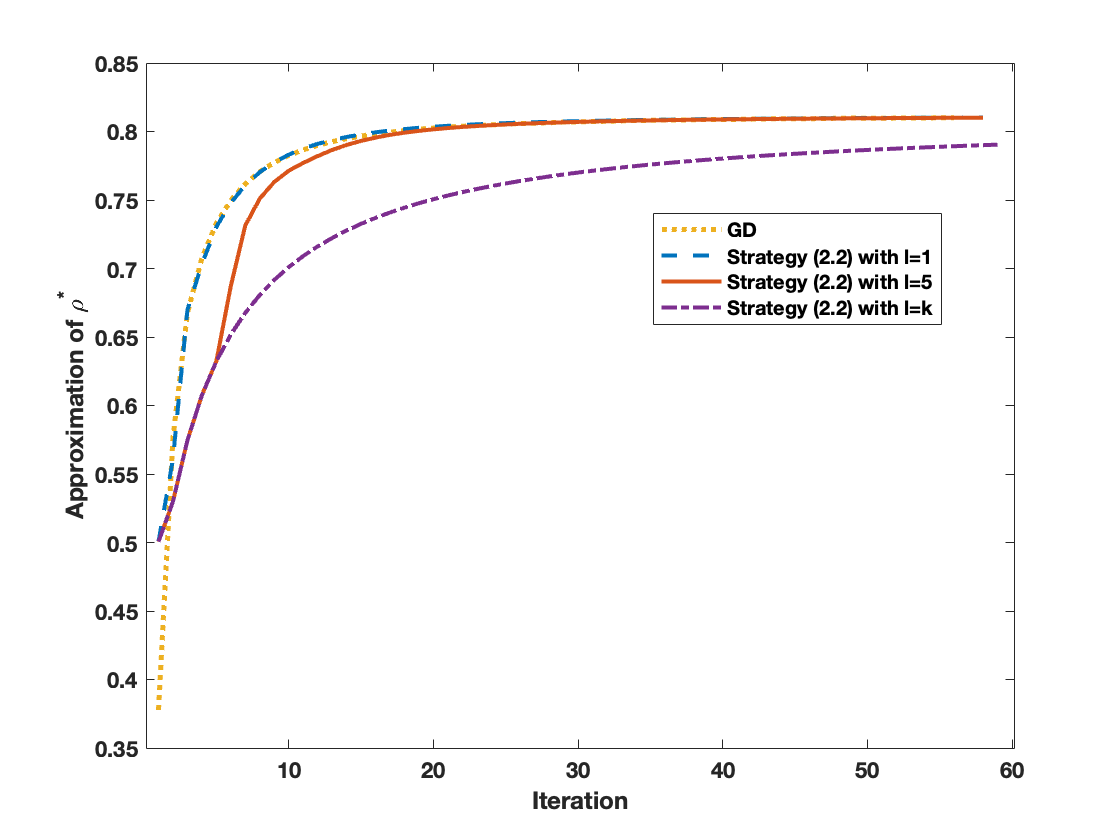}
    \end{minipage}
    \caption{Error (left) and estimated $\rho^*$ (right) for GD and its variants on the logistic regression problem.}
    \label{fig:logistic_gd_new}
\end{figure}

\begin{figure}[htbp]
    \centering
    \begin{minipage}[t]{0.49\textwidth}
        \centering
        \includegraphics[width=\textwidth]{ 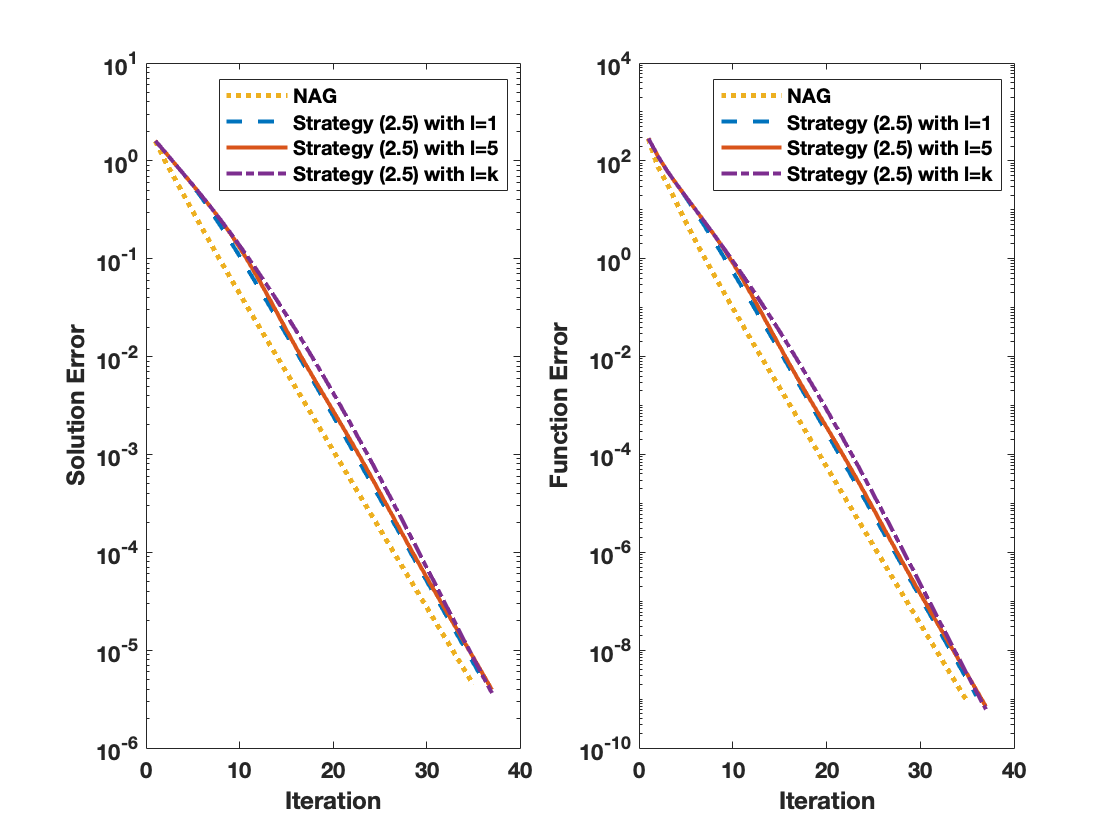}
    \end{minipage}%
    \hfill
    \begin{minipage}[t]{0.49\textwidth}
        \centering
        \includegraphics[width=\textwidth]{ 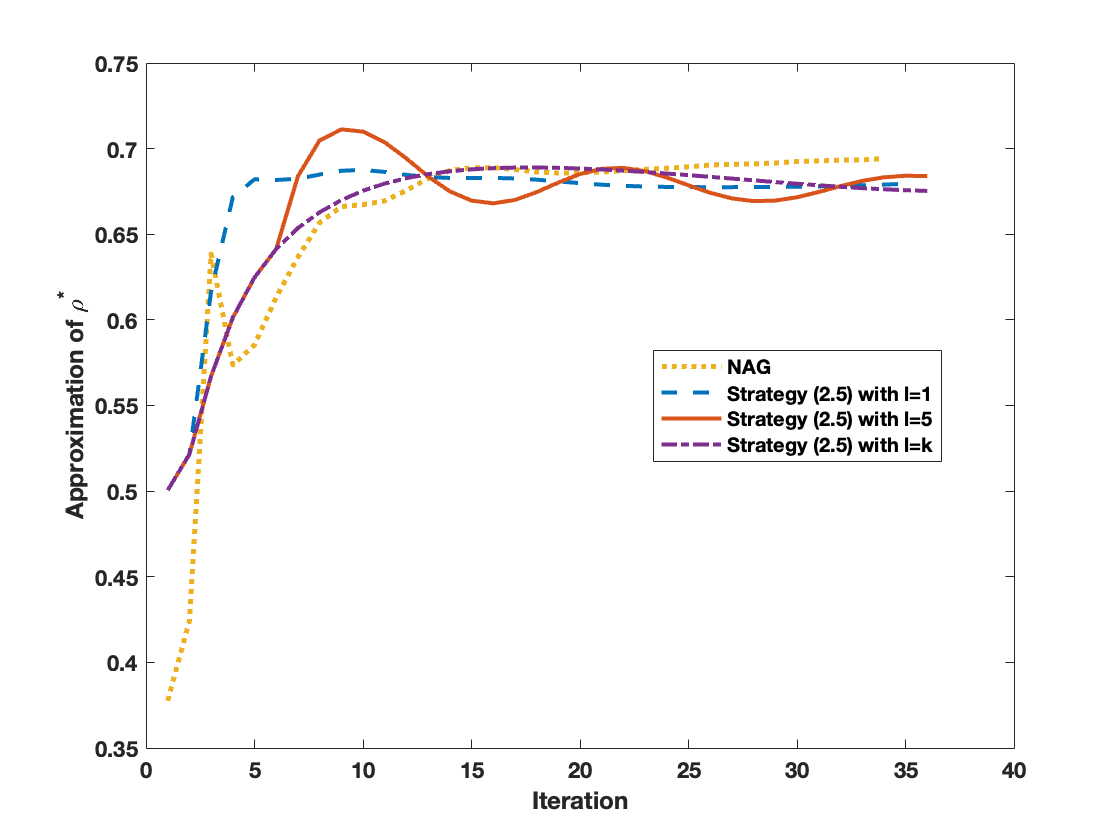}
    \end{minipage}
    \caption{Error (left) and estimated $\rho^*$ (right) for NAG and its variants on the logistic regression problem.}
    \label{fig:logistic_nag_new}
\end{figure}
 
\begin{figure}[htbp]
    \centering
    \begin{minipage}[t]{0.49\textwidth}
        \centering
        \includegraphics[width=\textwidth]{ 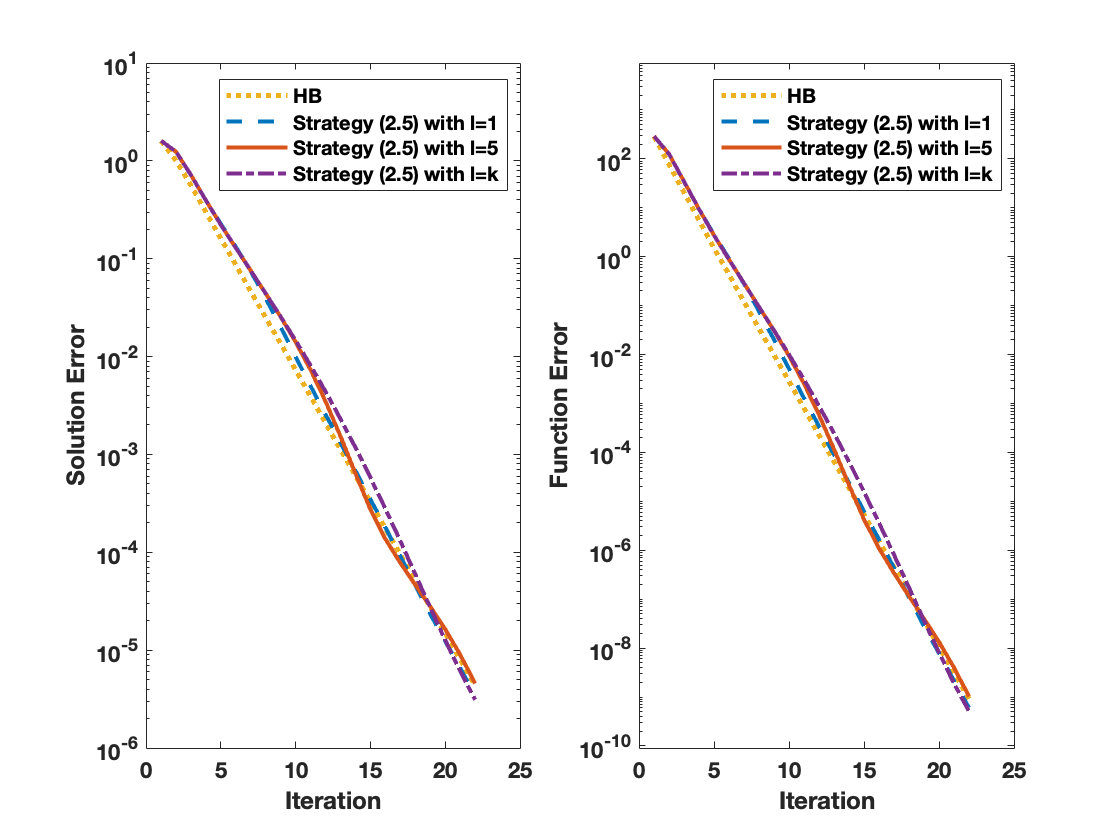}
    \end{minipage}%
    \hfill
    \begin{minipage}[t]{0.49\textwidth}
        \centering
        \includegraphics[width=\textwidth]{ 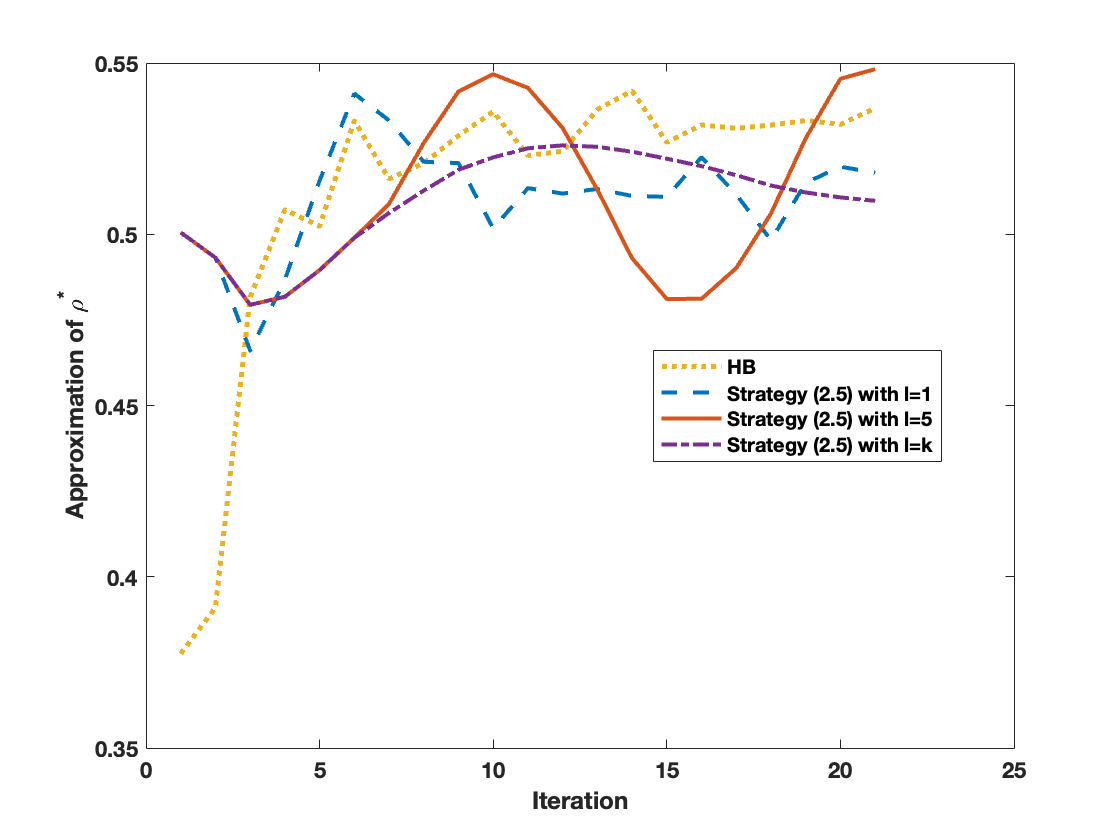}
    \end{minipage}
    \caption{Error (left) and estimated $\rho^*$ (right) for HB and its variants on the logistic regression problem.}
    \label{fig:logistic_hb_new}
\end{figure}

\subsection{Huber-TV regularized image denoising} Total Variation (TV) regularization \cite{rudin1992nonlinear} has been widely used in image processing \cite{chambolle2004algorithm,sidky2008image,lustig2007sparse}. The key idea is to promote piecewise constant images while preserving sharp edges. The classical TV model penalizes the $L^1$ norm of the image gradient, which leads to the undesirable staircasing effect \cite{nikolova2002minimizers}.  To address this, a common strategy is  replacing the non-smooth TV term with a Huber smoothing function:
\begin{align*}
h_\delta(t)=
\begin{cases}
\dfrac{t^{2}}{2\delta}, & \text{for } |t| \le \delta, \\[6pt]
|t| - \dfrac{\delta}{2}, & \text{for } |t| > \delta,
\end{cases}
\end{align*}
where $\delta>0$. Then, the  Huber-TV regularization is defined as $\int_\Omega h_\delta\big(|\nabla u(x)|\big)\,dx.$
This regularizer is often employed in image denoising \cite{chambolle2016introduction}. 
In this setting, one arrives at the following Huber-TV denoising model: 
\begin{align*}
\min_{u:\Omega \to \mathbb{R}} 
\; f(u)
=
\frac{\xi}{2}\int_{\Omega} (u(x) - u_0(x))^2 \, dx
+
\eta \int_{\Omega} h_\delta\!\left( |\nabla u(x)| \right) \, dx,
\end{align*}
where $u_0\in \mathbb{R}^{n \times m}$ is a given noisy image corrupted by Gaussian noise,  $\xi>0$ is the fidelity parameter, $\eta>0$ controls smoothness, and $|\nabla u(x)| = \sqrt{u_x(x)^2 + u_y(x)^2}$.  Then, the gradient of the objective can be written as
\begin{align*}
\nabla f(u)
=
\xi (u - u_0)
-
\eta\,\mathrm{div}\left( h'_{\delta}(|\nabla u|)\frac{\nabla u}{|\nabla u|} \right),
\end{align*}
For any $v \in H^1(\Omega)$, the Hessian of $f$ at $u$ applied to $v$
is given by 
\begin{align*}
\nabla^2 f(u)[v]
=\xi v-\eta\,\mathrm{div}\left(
\mathcal{A}(u)\nabla v
\right),
\end{align*}
where $\mathcal{A}(u)$ is defined  as
\begin{align*}
\mathcal{A}(u)
=
\begin{cases}
\dfrac{1}{\delta} I, & |\nabla u| \le \delta, \\[8pt]
\dfrac{1}{|\nabla u|}
\left(
I - \dfrac{\nabla u\,\nabla u^T}{|\nabla u|^2}
\right), & |\nabla u| > \delta.
\end{cases}
\end{align*}
Noting that when $|\nabla u| > \delta$, $I - \frac{\nabla u \nabla u^{\top}}{|\nabla u|^2}$ is an orthogonal projector and,  therefore, its $2$-norm equals $1$. Consequently, we have
$$
\left\| \frac{1}{|\nabla u|}
\left( I - \frac{\nabla u \nabla u^{\top}}{|\nabla u|^2} \right) \right\|
\le \frac{1}{|\nabla u|}
\le \frac{1}{\delta}.
$$
Hence, 
$\mathcal{A}(u)$ is uniformly bounded in the matrix 2-norm by $1/\delta$.

In the following tests, we consider an image $u_0 \in\mathbb{R}^{256\times256}$ corrupted by Gaussian noise with zero mean and variance $0.05^2$. We set $\xi = 4$, $\eta  = 0.06$ and $\delta  = 0.05$. 
The gradient $\nabla u$ is approximated using forward finite differences,
with Neumann boundary conditions. 
Thus the induced discrete Laplacian operator satisfies
$\|\Delta\|_{1}\le 8$.
Combining the above bounds, the strong convexity constant and the Lipschitz constant are estimated as $\tilde{\mu}=\xi$, $\tilde{L}=\xi+\frac{8\eta}{\delta}$, respectively. 



The results are summarized in   \Cref{fig:denosing_GD}-\ref{fig:denosing_HB}. 
We observe that the adaptive GD algorithms with strategy \eqref{average_GD} require a similar number of iterations as GD with fixed $\tilde L$ and $\tilde \mu$ to achieve comparable  objective decrease. This is because the estimated $L$ and $\mu$ remain close to the actual values across iterations. 
Moreover, the convergence factors $\rho^k$ for the adaptive algorithms are close to the one for  GD with fixed parameters. 
Similar results can also be observed for the adaptive NAG algorithms. For the HB method, using fixed parameters leads to fewer iterations under the same objective decrease. We also observe that, for the method with fixed parameters, the estimated convergence factor $\rho_k$ gradually approaches $(\tilde L-\tilde\mu)/(\tilde L+\tilde\mu)$.
However, for the adaptive algorithms, $\rho_k$ stabilizes around $0.7$. This behavior may result from the difference in initialization: the adaptive algorithms start with GD using step size $1/\tilde{L}$, while the fixed-parameter approach initializes GD with step size $1/(\tilde{L}+\tilde{\mu})$. This difference is further amplified by the parameter sensitivity of the HB method.

\begin{figure}[htbp]
    \centering
    \begin{minipage}[t]{0.49\textwidth}
        \centering
        \includegraphics[width=\textwidth]{ 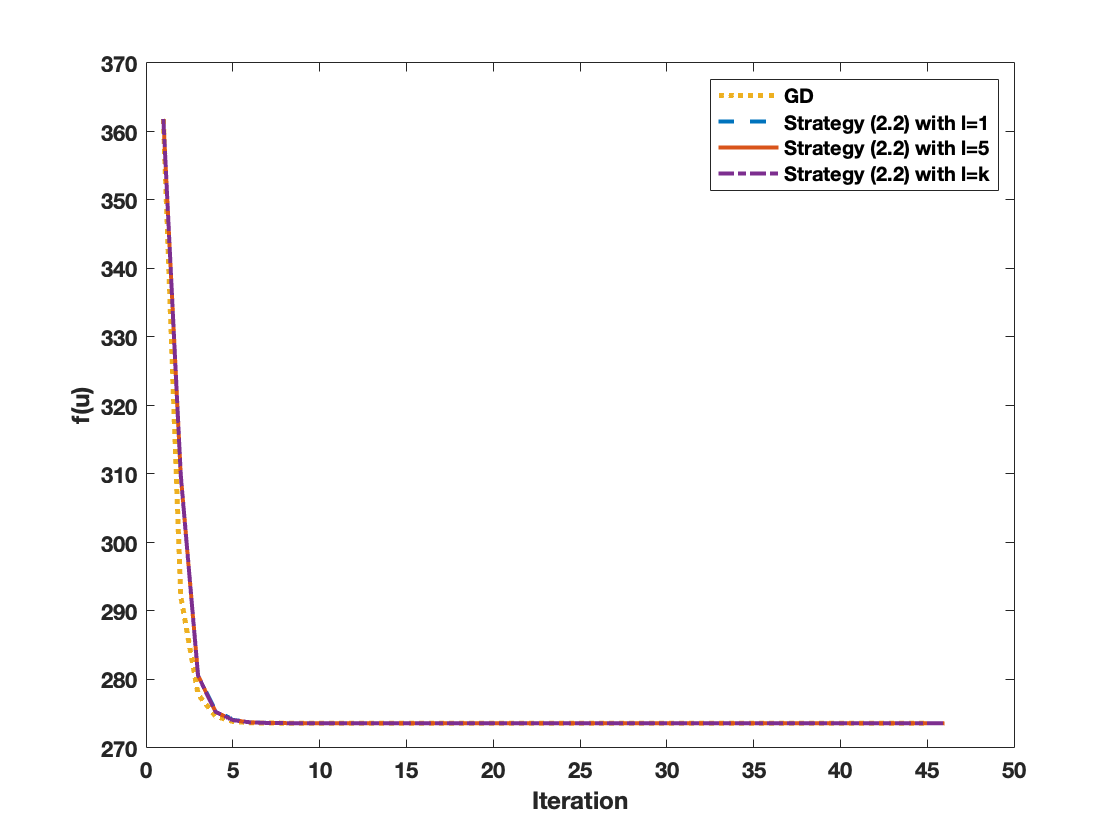}
    \end{minipage}%
    \hfill
    \begin{minipage}[t]{0.49\textwidth}
        \centering
        \includegraphics[width=\textwidth]{ 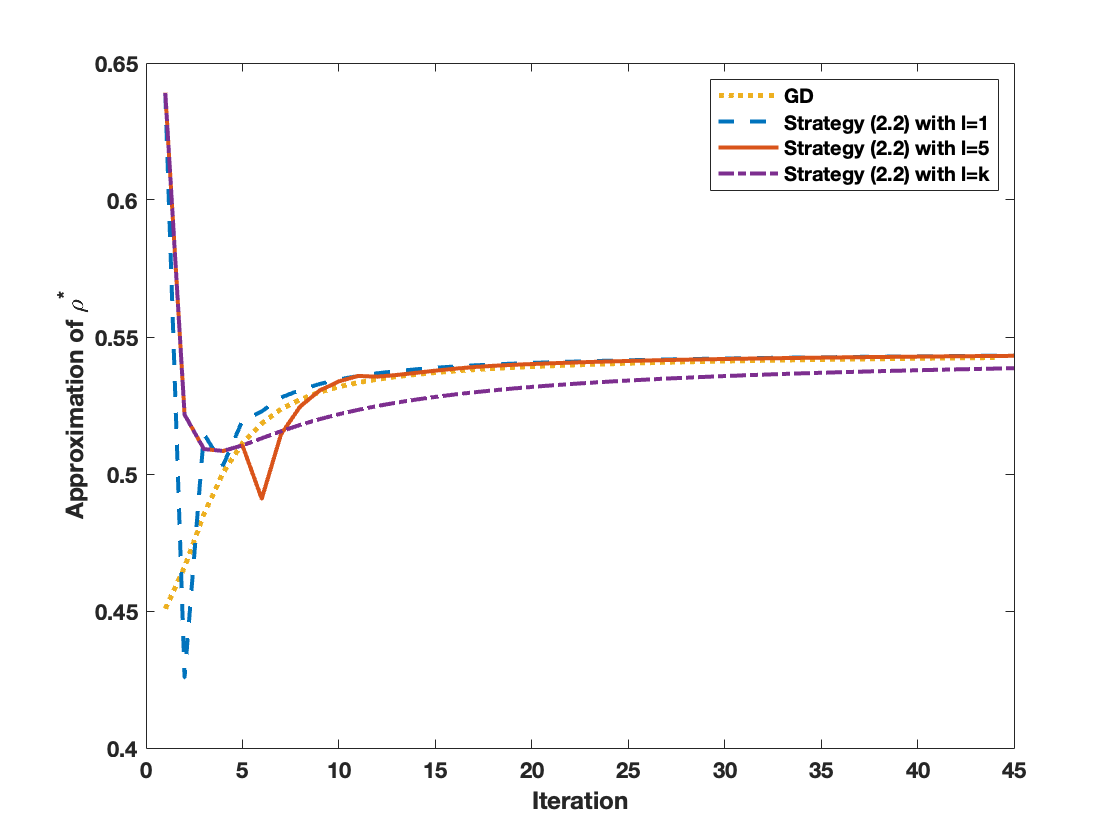}
    \end{minipage}
    \caption{Decay of the objective value (left) and estimated $\rho^*$ (right) for GD and its variants on the image denoising problem.}
    \label{fig:denosing_GD}
\end{figure}

\begin{figure}[htbp]
    \centering
    \begin{minipage}[t]{0.49\textwidth}
        \centering
        \includegraphics[width=\textwidth]{ 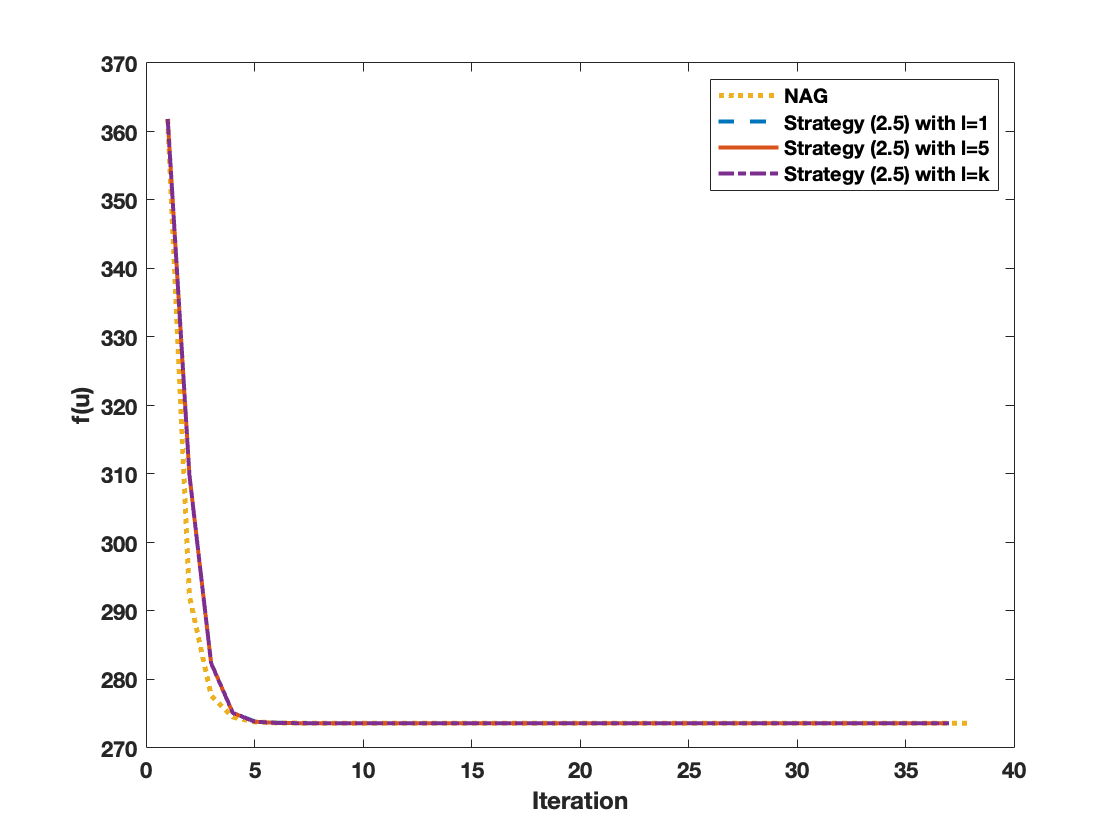}
    \end{minipage}%
    \hfill
    \begin{minipage}[t]{0.49\textwidth}
        \centering
        \includegraphics[width=\textwidth]{ 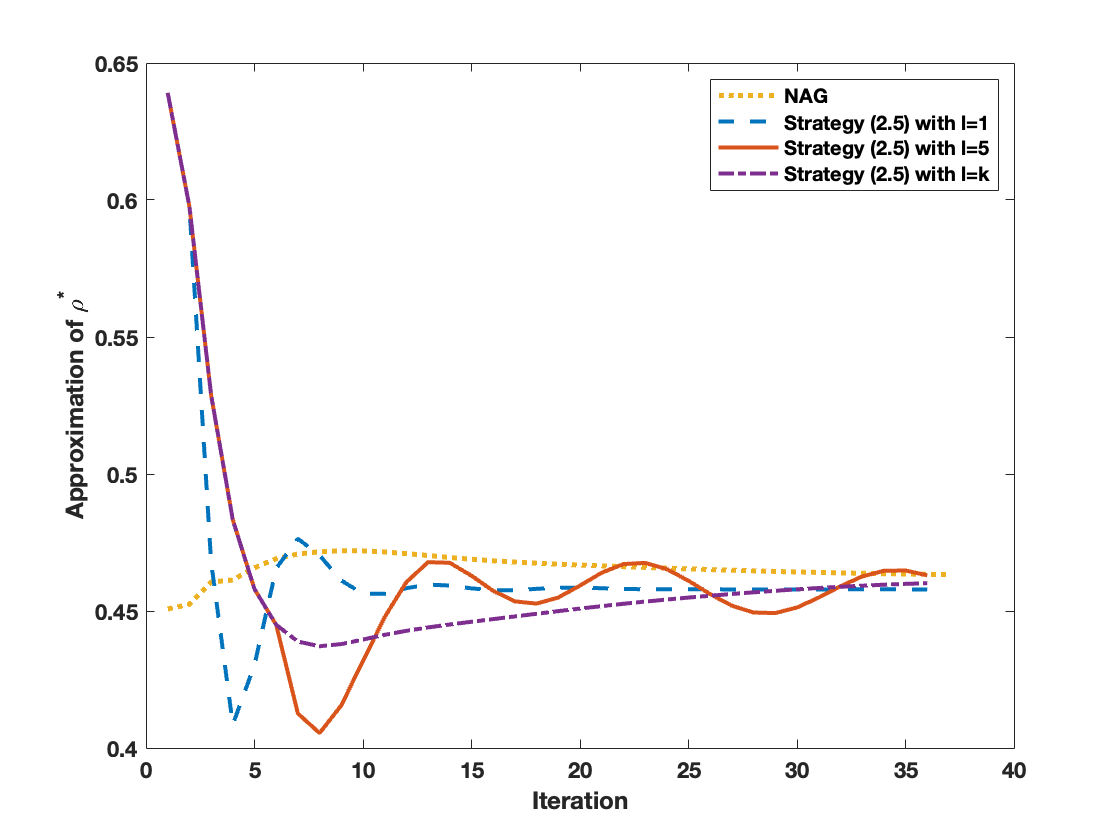}
    \end{minipage}
    \caption{Decay of the objective value (left) and estimated $\rho^*$ (right) for NAG and its variants on the image denoising  problem.}
    \label{fig:denosing_nag}
\end{figure}
 
\begin{figure}[htbp]
    \centering
    \begin{minipage}[t]{0.49\textwidth}
        \centering
        \includegraphics[width=\textwidth]{ 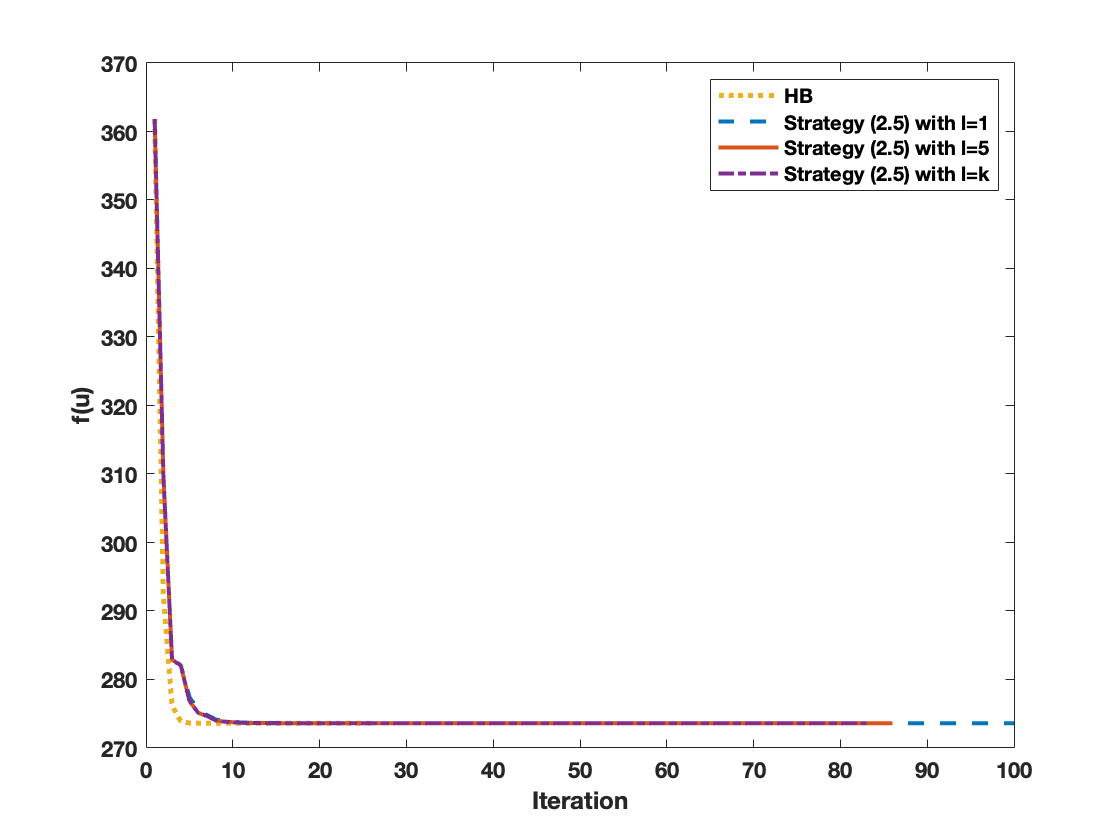}
    \end{minipage}%
    \hfill
    \begin{minipage}[t]{0.49\textwidth}
        \centering
        \includegraphics[width=\textwidth]{ 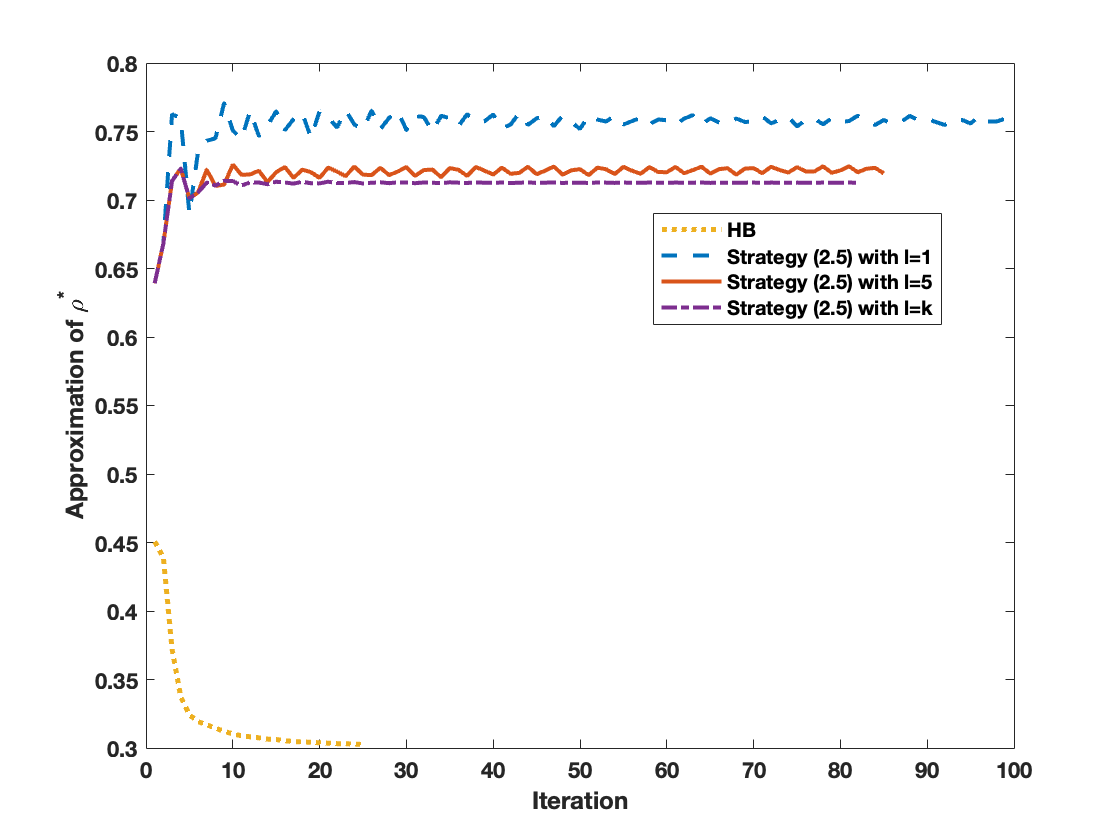}
    \end{minipage}
    \caption{Decay of the objective value (left) and estimated $\rho^*$ (right) for HB and its variants on the image denoising  problem.}
    \label{fig:denosing_HB}
\end{figure}

\section{Conclusion}\label{conclusion}
By utilizing the ratio of residual norms at two consecutive iterations as an empirical estimate of the upper bounds on convergence rates, we propose an adaptive framework for tuning the step size and momentum parameters in first-order methods for unconstrained quadratic optimization problems. We establish that the sequence of iterates generated by these adaptive gradient methods converges to $\bm{x}^*$ at a rate at least as favorable as that of GD with a step size of $1/L$. 

Numerical results on both quadratic and general strongly convex optimization problems demonstrate the effectiveness of our proposed methods, showing that the adaptive algorithms achieve efficiency comparable to their accelerated  counterparts with optimal parameters. However, as noted, a gap remains between the spectral radius and the matrix $2$-norm in our current analysis. To bridge this gap, a potential direction for future research involves utilizing Lyapunov functions \cite{wilson2021lyapunov,wibisono2016variational} to construct adaptive algorithms and provide the corresponding theoretical analysis.

\section*{Acknowledgments}
The authors wish to thank James H. Adler for many insightful and helpful discussions and suggestions.

\bibliographystyle{siamplain}
\bibliography{references}
\end{document}